\documentclass[11pt]{article}

\usepackage{mathrsfs}

\usepackage[T1]{fontenc}
\usepackage[utf8]{inputenc}
\usepackage{lmodern}
\usepackage{amssymb,amsmath,amsthm}
\usepackage{bbm}
\usepackage{graphicx}
\usepackage{float}
\usepackage{xcolor}
\usepackage[colorlinks,linkcolor=black!70!black,citecolor=black]{hyperref} 
\usepackage[a4paper,margin=1.5cm]{geometry}
\usepackage{tikz}
\usepackage{tikz-cd}

\usetikzlibrary{decorations.pathreplacing, patterns,shapes,snakes}

\usepackage{pgfplots}

\long\def\metanote#1#2{{\color{#1}\
\ifmmode\hbox\fi{\sffamily\mdseries\upshape [#2]}\ }}

\newcommand{\ra}{\rightarrow}

\newcommand{\be}{\begin{equation}}
\newcommand{\ee}{\end{equation}}
\newcommand{\bi}{\begin{itemize}}
\newcommand{\ei}{\end{itemize}}

\newcommand{\commentout}[1]{}

\newcommand{\bfQ}{{\bf{Q}}}

\newcommand{\TV}{\text{TV}}

\newcommand{\Wah}{\text{W}}

\newcommand{\Law}{{\mathcal{L}}}

\newcommand{\calM}{{\mathcal{M}}}

\newcommand{\Leb}{{\text{Leb}}}

\newcommand{\calB}{{\mathcal{B}}}

\newcommand{\calD}{{\mathcal{D}}}

\newcommand{\calO}{{\mathcal{O}}}

\newcommand{\calP}{{\mathcal{P}}}

\newcommand{\Nm}{{\mathbb{N}}}

\newcommand{\Rm}{{\mathbb R}}

\newcommand{\Pm}{{\mathbb P}}

\newcommand{\expE}{{\mathbb E}}
\newcommand{\Ind}{{\mathbbm{1}}}

\newtheorem{theo}{Theorem}[section]
\newtheorem{lem}[theo]{Lemma}
\newtheorem{defin}[theo]{Definition}

\newtheorem{prop}[theo]{Proposition}

\newtheorem*{rep@theorem}{\rep@title}
\newcommand{\newreptheorem}[2]{%
\newenvironment{rep#1}[1]{%
 \def\rep@title{#2 \ref{##1}}%
 \begin{rep@theorem}}%
 {\end{rep@theorem}}}
\makeatother
\newreptheorem{theorem}{Theorem}
\newreptheorem{lemma}{Lemma}
\newreptheorem{defin}{Definition}
\newreptheorem{cond}{Condition}
\newreptheorem{prop}{Proposition}
\newreptheorem{lem}{Lemma}
\newreptheorem{cor}{Corollary}
\newreptheorem{rmk}{Remark}
\newreptheorem{exer}{Exercise}
\newreptheorem{conj}{Conjecture}
\newreptheorem{assum}{Assumption}
\newreptheorem{equation}{Equation}

\long\def\metanote#1#2{{\color{#1}\
\ifmmode\hbox\fi{\sffamily\mdseries\upshape [#2]}\ }}

\begin{document}
\setcounter{page}{1}

\title{Selection principle for the Fleming-Viot process with drift $-1$}
\author{Oliver Tough\footnote{Department of Mathematical Sciences, University of Bath.}}

\date{June 5, 2023}

\maketitle

\begin{abstract}
We consider the Fleming-Viot particle system consisting of $N$ identical particles evolving in $\mathbb{R}_{>0}$ as Brownian motions with constant drift $-1$. Whenever a particle hits $0$, it jumps onto another particle in the interior. It is known that this particle system has a hydrodynamic limit as $N\rightarrow\infty$ given by Brownian motion with drift $-1$ conditioned not to hit $0$. This killed Brownian motion has an infinite family of quasi-stationary distributions (QSDs), with a Yaglom limit given by the unique QSD minimising the survival probability. On the other hand, for fixed $N<\infty$, this particle system converges to a unique stationary distribution as time $t\rightarrow\infty$. We prove the following selection principle: the empirical measure of the $N$-particle stationary distribution converges to the aforedescribed Yaglom limit as $N\rightarrow\infty$. The selection problem for this particular Fleming-Viot process is closely connected to the microscopic selection problem in front propagation, in particular for the $N$-branching Brownian motion. The proof requires neither fine estimates on the particle system nor the use of Lyapunov functions.
\end{abstract}

\section{Introduction and main result}\label{section:introduction}

The long-term behaviour of Markov processes with an absorbing boundary has been studied since the work of Yaglom \cite{Yaglom1947} on subcritical Galton-Watson processes. The limits we obtain are known as quasi-stationary distributions (QSDs). Given an absorbed Markov process $(X_t)_{0\leq t<\tau_{\partial}}$, QSDs are probability measures $\pi$ such that
\begin{equation}
\Law_{\pi}(X_t\lvert \tau_{\partial}>t)=\pi\quad\text{for all}\quad t\geq 0,\quad\text{so that}\quad\Pm_{\pi}(X_t\in \cdot, \tau_{\partial}>t)=e^{-\lambda(\pi)t}\pi,
\end{equation}
for some constant $\lambda(\pi)\geq 0$. We refer to the constant $\lambda(\pi)$ as the eigenvalue of $\pi$, since QSDs are left eigenmeasures of the infinitesimal generator of eigenvalue $-\lambda(\pi)$ (see \cite[Proposition 4, p.349]{Meleard2011}). A survey of QSDs is provided by \cite{Meleard2011}, due to M\'{e}l\'{e}ard and Villemonais. For processes on a bounded domain, it is typical for there to be a unique QSD, because the transition semigroup is typically compact and irreducible. For processes in an unbounded domain, however, it is typical for there to be infinitely many QSDs.

In this article we focus on Brownian motion on the positive half-line $\Rm_{>0}$ with constant drift $-1$, killed instantaneously at $0$,
\begin{equation}\label{eq:BM with drift -1}
dX_t=-dt+dW_t,\quad 0\leq t<\tau_{\partial}:=\inf\{t>0:X_{t-}=0\}.
\end{equation}

Mart{\'{i}}nez, Picco and San Mart{\'{i}}n \cite{Martinez1994} provided the following classification of the quasi-stationary distributions of $(X_t)_{0\leq t<\tau_{\partial}}$. The QSDs of $(X_t)_{0\leq t<\tau_{\partial}}$ are given by a one-parameter family of probability measures $(\pi_{\lambda})_{0<{\lambda}\leq \frac{1}{2}}$, with associated eigenvalues $\lambda(\pi_{\lambda})=\lambda$, given by
\begin{equation}\label{eq:QSDs of drift -1}
\pi_{\lambda}(dx)=\begin{cases}
M_{\frac{1}{2}}xe^{- x}dx,\quad \lambda=\frac{1}{2}\\
M_{\lambda} e^{-x}\sinh(\sqrt{1-2\lambda}x)dx,\quad 0<\lambda<\frac{1}{2}
\end{cases},\quad\text{for all}\quad 0<\lambda\leq \frac{1}{2}.
\end{equation}
In the above $M_{\lambda}$ for $0<{\lambda}\leq \frac{1}{2}$ are normalisation constants. We refer to $\pi_{\frac{1}{2}}$ as the \text{minimal QSD}, and will denote it as $\pi_{\min}$. In general, a QSD with eigenvalue $\lambda(\pi)$ will have killing time given by $\Law_{\pi}(\tau_{\partial})=\text{exp}(\lambda(\pi))$, so that $\expE_{\pi}[\tau_{\partial}]=\frac{1}{\lambda(\pi)}$. The QSD maximising $\lambda(\pi)$, in this case $\pi_{\min}$, is called \textit{minimal} because it minimises $\expE_{\pi}[\tau_{\partial}]=\frac{1}{\lambda(\pi)}$.

Mart{\'{i}}nez, Picco and San Martin \cite{Martinez1998} established a sufficient condition for a probability measure to be in the domain of attraction of the minimal QSD for \eqref{eq:BM with drift -1}, $\pi_{\min}$. The following, proven in the appendix, establishes that this is necessary and sufficient.
\begin{theo}[Theorem 1.3, \cite{Martinez1998}]\label{theo:domain of attraction of minimal QSD}
We suppose that $\mu$ is a probability measure on $\Rm_{>0}$. Then the following are equivalent:
\begin{align}
\Law_{\mu}(X_t\lvert \tau_{\partial}>t)\overset{\TV}{\ra} \pi_{\min}\quad\text{as}\quad t\ra\infty,\label{eq:convergence to minimal QSD}\\
\limsup_{t\ra\infty}\frac{1}{t}\ln\Pm_{\mu}(\tau_{\partial}>t)\leq -\lambda_{\min}=\frac{1}{2},\label{eq:limsup of survival prob bded}\\
\limsup_{x\ra \infty}\frac{1}{x}\ln\mu([x,\infty))\leq -1,\label{eq:Martinez condition for domain of attraction}\\
\int_{\Rm_{>0}}e^{ux}\mu(dx)<\infty\quad\text{for all} \quad u<1.\label{eq:condition for domain of attraction of minimal QSD}
\end{align}
\end{theo}
In particular, we have the convergence in \eqref{eq:convergence to minimal QSD} whenever the initial condition $\mu$ is a point mass, meaning that $\pi_{\min}$ is a \textit{Yaglom limit} (see \cite[Definition 2]{Meleard2011}).

In this paper we study the behaviour of a system of interacting particles associated to $(X_t)_{0\leq t<\tau_{\partial}}$, known as a Fleming-Viot particle system. This is defined as follows. 
\begin{defin}[Fleming-Viot particle system]
We consider $N \geq 2$ particles diffusing in the domain $\Rm_{>0}$. The particle positions are denoted by $X_t^{N,1},\ldots,X_t^{N,N} \in \Rm_{>0}$, giving an $\Rm_{>0}^N$-valued stochastic process $\vec{X}^N_t:=(X_t^{N,1},\ldots,X_t^{N,N})$. We let $\upsilon^N$ be a probability measure on $\Rm_{>0}^N$, and let $\{ W^{N,i}_t\}_{i=1}^N$ be a collection of independent Brownian motions on $\Rm^d$. Then the particle system $\{X^{N,i}\}_{i=1}^N \subset \Rm_{>0}$ with initial distribution $\upsilon^N$ is defined by:
\begin{equation} \label{eq:N-particle system sde}
\left \{ \begin{split}
(i) & \quad \vec{X}^N_0 \sim \upsilon^N. \\
(ii) & \quad \text{For $t \in [0,\infty)$ and between jump times, the particles evolve according to}\\
&  \quad dX^{N,i}_t=-dt+dW^{N,i}_t,\quad i = 1,\dots,N. \\
(iii) & \quad \text{Whenever a particle $X^{N,i}$ hits the boundary $0$, i.e. $X^{N,i}_{t-}=0$, $X^{N,i}$ instantly jumps to the}\\
& \quad \text{location of another particle chosen independently and uniformly at random.}
\end{split}  \right.
\end{equation}
We denote by $\tau_n^N$ the $n^{\text{th}}$ jump time of the particle system for $n\geq 1$, with $\tau^N_0:=0$. We define
\begin{equation}\label{eq:jump emp meas processes}
J^N_t:=\frac{1}{N}\lvert \{n:\tau_n^N\leq t\}\rvert,\quad m^N_t:=\frac{1}{N}\sum_{i=1}^N\delta_{X^{N,i}_t}(\cdot).
\end{equation}
Thus $J^N_t$ is the number of jumps up to time $t$, renormalised by $\frac{1}{N}$, and $m^N_t $ is the empirical measure at time $t$.
\label{defin:Fleming-Viot MKV dynamics}
\end{defin}
In general, it is possible for the Fleming-Viot particle system to be ill-posed, due to the possibility of there being infinitely many jumps in a finite time, as in \cite{Bieniek2012}. Well-posedness is established in the present setting by \cite[Theorem 3.6]{Villemonais2011}.

Villemonais has established a hydrodynamic limit theorem under general conditions \cite[Theorem 2.2]{Villemonais2011}, which in this case provides the following.
\begin{theo}[Theorem 2.2, \cite{Villemonais2011}]
We consider a probability measure $\mu\in\calP(\Rm_{>0})$ and a sequence of $N$-particle Fleming-Viot processes $(\vec{X}^N_t)_{t\geq 0}$ such that $m^N_0\ra \mu$ weakly in probability as $N\ra\infty$. Then for all $t>0$ we have that $m^N_t\ra \Law_{\mu}(X_t\lvert \tau_{\partial}>t)$ weakly in probability as $N\ra\infty$.
\end{theo}

If $\mu$ satisfies \eqref{eq:Martinez condition for domain of attraction} then Theorem \ref{theo:domain of attraction of minimal QSD} ensures that this limit, $\Law_{\mu}(X_t\lvert \tau_{\partial}>t)$, converges to the minimal QSD, $\pi_{\min}$, as $t\ra\infty$. On the other hand it is straightforward to prove the following.
\begin{theo}\label{theo:ergodicity for fixed N}
We fix $2\leq N<\infty$. Then there exists a unique stationary distribution $\psi^N\in \calP(\Rm_{>0}^N)$ such that 
\begin{equation}\label{eq:convergence in TV of Law fixed N}
\Law_{\upsilon^N}(X_t)\ra \psi^N\quad\text{in total variation as}\quad t\ra \infty,
\end{equation}
for any initial distribution $\vec{X}^N_0\sim \upsilon^N\in \calP(\Rm_{>0}^N)$.
Moreoer there exists a constant $\lambda_N\in (0,\infty]$ such that
\begin{equation}\label{eq:almost sure convergence of lambda N fixed N}
\frac{J^N_t}{t}\ra \lambda_N\quad\text{almost surely as}\quad t\ra\infty,
\end{equation}
for any initial distribution $\vec{X}^N_0\sim \upsilon^N\in \calP(\Rm_{>0}^N)$. Furthermore $\lambda_N\in (0,\infty)$ for $N\geq 12$. Finally, the convergence in \eqref{eq:almost sure convergence of lambda N fixed N} becomes $L^p$ convergence whenever the initial condition is deterministic and $1\leq p<\frac{1}{2}\lfloor \frac{N}{4}\rfloor$. 
\end{theo}

A proof of Theorem \ref{theo:ergodicity for fixed N} when the Fleming-Viot particle system is driven instead by a biased random walk on $\Nm$ is given in \cite{Asselah2012}. We supply a proof of Theorem \ref{theo:ergodicity for fixed N} in the appendix. 

\subsection{Main result}

We let $\mathcal{P}(\Rm_{>0})$ be the set of Borel probability measures on $\Rm_{>0}$, and let $\Theta^N:\Rm_{>0}^N \to \mathcal{P}(\Rm_{>0})$ be the map which takes the points $x_1,\dots,x_N \in \Rm_{>0}$ to their empirical measure,
\begin{equation}
 \Theta^N(x_1,\ldots,x_N) = \frac{1}{N}\sum_{k=1}^N\delta_{x_k},
\label{eq:empirical map}
\end{equation}
which is invariant under permutation of the indices. 

We consider for each $N<\infty$ the stationary empirical measure of the $N$-particle Fleming-Viot particle system, 
\begin{equation}
\chi^N:=\Theta^N_{\#}\psi^N=\Law_{\vec{X}^N\sim \psi^N}(\frac{1}{N}\sum_{i=1}^N\delta_{X^{N,i}}).
\end{equation}
This is the law of the empirical measure of the $N$-particle Fleming-Viot process when the process is distributed according to its stationary distribution $\psi^N$. 

Our main theorem is the following selection principle.
\begin{theo}[Selection principle for the Fleming-Viot process]\label{theo:main theorem}
We have the convergences
\begin{align}
\chi^N\ra \pi_{\min}\quad \text{weakly in probability and}\quad \lambda_N\ra \lambda_{\min}\quad \text{as}\quad N\ra\infty.
\end{align}
\end{theo}
The following diagram summarises the relationship between the above.

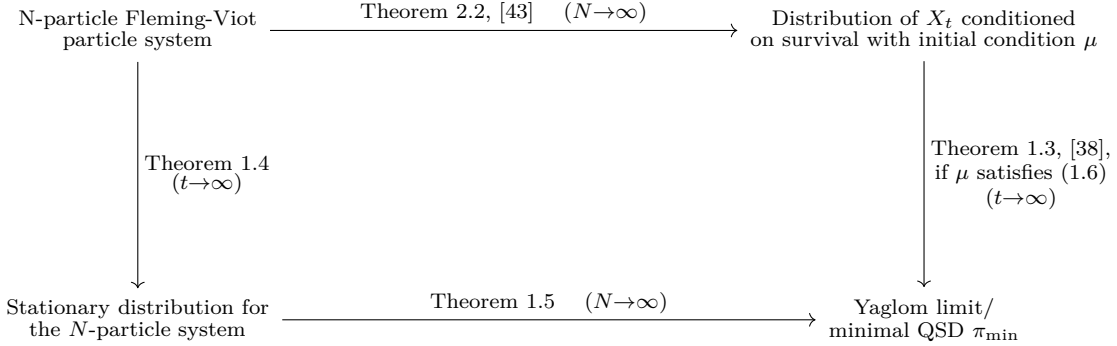
\begin{tikzcd}[column sep=6cm, row sep=3cm]
\substack{\text{N-particle Fleming-Viot}\\\text{particle system}} \arrow[r, "\text{Theorem 2.2, \cite{Villemonais2011}}\quad (N\ra\infty)"]  \arrow[d, "\substack{\text{Theorem }\ref{theo:ergodicity for fixed N}\\(t\ra\infty)}"]
& \substack{\text{Distribution of $X_t$ conditioned}\\ \text{on survival with initial condition $\mu$}} \arrow[d, "\substack{\text{Theorem 1.3, \cite{Martinez1998},}\\ \text{if $\mu$ satisfies \eqref{eq:Martinez condition for domain of attraction}}\\ (t\ra \infty)}"] \\
\substack{\text{Stationary distribution for}\\ \text{the $N$-particle system}} \arrow[r, "\text{Theorem }\ref{theo:main theorem}\quad  (N\ra\infty)"]
& \substack{\text{Yaglom limit/}\\ \text{minimal QSD $\pi_{\min}$}}
\end{tikzcd}

If the weak limit in probability of the initial empirical measures, denoted as $\mu$, does not belong to the domain of attraction of the minimal QSD, then these limits will not commute. Consider, for example, a sequence of Fleming-Viot $N$-particle systems such that for some $\lambda\in (0,\frac{1}{2})$,
\[
m^N_0 \ra \pi_{\lambda}\quad\text{weakly in probability.}
\] 
The above diagram instead becomes the following.

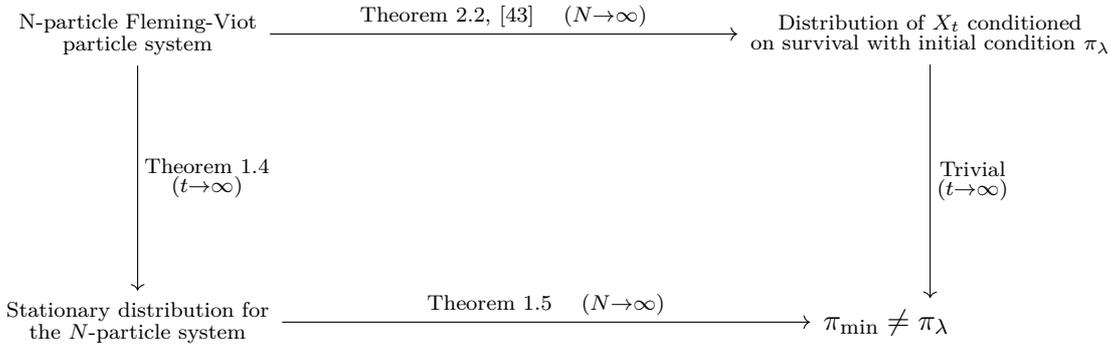
\begin{tikzcd}[column sep=6cm, row sep=3cm]
\substack{\text{N-particle Fleming-Viot}\\\text{particle system}} \arrow[r, "\text{Theorem 2.2, \cite{Villemonais2011}}\quad (N\ra\infty)"]  \arrow[d, "\substack{\text{Theorem }\ref{theo:ergodicity for fixed N}\\(t\ra\infty)}"]
& \substack{\text{Distribution of $X_t$ conditioned}\\ \text{on survival with initial condition $\pi_{\lambda}$}} \arrow[d, "\substack{\text{Trivial}\\ (t\ra \infty)}"] \\
\substack{\text{Stationary distribution for}\\ \text{the $N$-particle system}} \arrow[r, "\text{Theorem }\ref{theo:main theorem}\quad  (N\ra\infty)"]
&\pi_{\min}\neq \pi_{\lambda} \quad\quad\quad
\end{tikzcd}

\subsection{Background and related work}

\subsubsection*{The selection problem for the Fleming-Viot process}

The Fleming-Viot particle system was first introduced by Burdzy, Hołyst and March \cite{Burdzy2000} to provide a particle representation for quasi-stationary distributions. Their work involved the particular case of purely Brownian dynamics in a bounded domain. They established that the empirical measure of the $N$-particle stationary distribution converges as $N\ra\infty$ to the unique QSD, as in Theorem \ref{theo:main theorem}. This was later extended to allow for more general domains in \cite[Theorem 7.2]{Bieniek2009}. A similar result was established by Ferrari and Maric \cite[Theorem 1.4]{Ferrari2006} in countable state spaces under a Dobrushin-type condition ensuring in particular uniqueness of the QSD and fast return from infinity. The finite state space case, for which we have a unique QSD by classical results of Darroch and Seneta \cite{Darroch1965,Darroch1967}, was dealt with by Ferrari and Asselah in \cite{Asselah2011}. Finally, the McKean-Vlasov setting was considered by Nolen and the present author in \cite{Tough2022}. It was established that if the corresponding killed McKean-Vlasov process conditioned on survival converges to a unique QSD, then the empirical measure of the $N$-particle stationary distribution also converges to this unique QSD as $N\ra\infty$ \cite[Theorem 2.16]{Tough2022}.

The aforementioned results were restricted to settings for which there is a unique QSD. Settings in which we have infinitely many QSDs have been considered by Villemonais in \cite{Villemonais2015}, Asselah, Ferrari, Groisman and Jonckeere in \cite{Asselah2016}, and Champagnat and Villemonais in \cite{Champagnat2021a}. The first considered birth-death processes on $\Nm$ satisfying a Lyapunov-type condition. In the second the dynamics are given by a subcritical Galton-Watson process on $\Nm$ which is killed at $0$. The latter considered processes with soft killing (i.e. according to a position dependent Poisson clock), where the killing rate is everywhere less than the exponential rate of convergence of the unkilled process. These established that the empirical measure of the $N$-particle stationary distribution converges as $N\ra\infty$ to the corresponding minimal QSD, under various assumptions.

Whereas this article concerns itself with the selection problem for the Fleming-Viot process driven by Brownian motion with constant drift $-1$, there is an identical problem for the biased random walk on $\Nm$ with constant bias towards $0$, killed instantaneously at $0$. For this identical problem, Maric provided numerical evidence that the selection principle should be true in \cite{Maric2015}, but nothing has been proven about this selection problem. In order to extend the results of this paper to the Fleming-Viot process driven by a biased random walk on $\Nm$, the only missing piece is a characterisation of the domain of attraction of the Yaglom limit akin to \cite[Theorem 1.3]{Martinez1998}. Whilst one would imagine it should be the same, to the authors' knowledge this has not yet been done.

Nevertheless, despite their apparent simplicity, the selection problem for the Fleming-Viot process driven by Brownian motion with constant negative drift, and the identical problem for the biased random walk on $\Nm$, have proven surprisingly difficult. The selection principles established in \cite{Villemonais2015}, \cite{Asselah2016} and \cite{Champagnat2021a} are obtained by constructing Lyapunov functions for the $N$-particle system, which are required to be uniform in $N$. However, in the case of the biased random walk on $\Nm$, Villemonais has established in \cite[Theorem 2.5]{Villemonais2015} that the existence of a suitable Lyapunov function would imply $\xi_1$-positive recurrence, which is known to be untrue (see \cite[Remark 3.3]{Villemonais2015}). Therefore a Lyapunov function approach would seem to have no hope of being successful in the present setting.

In general, in order to establish convergence of the stationary empirical measures $\chi^N$ to a given QSD $\pi$, it suffices to establish tightness of the $\chi^N$, and that subsequential limits are supported on the domain of attraction of $\pi$. The first difficulty is to establish tightness. The particle system spreads out every time a particle is killed, since this particle jumps from $0$ to one of the other particles. Counteracting this is only a weak $-1$ drift pushing the particles to $0$. This makes it difficult to obtain estimates on the particles providing for tightness, which was unknown prior to the present article. Moreover, examining \eqref{eq:QSDs of drift -1}, we see that the minimal QSD $\pi_{\min}=\pi_{\frac{1}{2}}$ is very similar to the QSD $\pi_{\frac{1}{2}-\epsilon}$ for small $\epsilon>0$. In fact, whilst the minimal QSD is given by $\pi_{\min}=M_{\frac{1}{2}}xe^{-x}$, \cite[Theorem 1.4]{Martinez1998} provides for probability measures $\mu$ and arbitrarily small  $\epsilon>0$ such that $\int_{\Rm_{>0}}e^{(1-\epsilon)x}\mu(dx)<\infty$ whilst $\mu$ does not belong to the domain of attraction of any QSD. Therefore controls distinguishing subsequential limits as $\pi_{\min}$ would have to be sharp. We may contrast this with the case of Ornstein-Uhlenbeck dynamics, for which the minimal QSD is $Mxe^{-x^2}$ whilst all other QSDs have polynomial tails (this follows from \cite[(23)]{Lladser2000}). We may also contrast the domain of attraction of the minimal QSD in the present setting with the Galton-Watson setting considered in \cite{Asselah2016}, where the domain of attraction of the minimal QSD includes all probability measures with finite first moment \cite[Theorem 6, p.352]{Meleard2011}.

Our proof will require neither fine controls on the particle system nor the use of Lyapunov functions.

\subsubsection*{The selection problem in front propagation}

The selection problem arose in the context of front propagation. The reader is directed towards \cite{Groisman2019}, due to Groisman and Jonckeere, for a survey of the relationship between the selection problem for QSDs and the Fleming-Viot process, and the selection problem in front propagation.

The Fisher-KPP equation,
\begin{equation}\label{eq:FKPP equation}
\frac{\partial u}{\partial t}=\frac{\partial^2 u}{\partial x^2}+u(1-u),
\end{equation}
was introduced independently in 1937 by Fisher \cite{Fisher1937} and Kolmogorov, Petrovskii and Piskunov \cite{Kolmogorov1937} as a model for the spatial spread of an advantageous allele. It was independently shown by both to have an infinite family of travelling wave solutions - solutions of the form $u_c(t,x)=w_c(x-ct)$ - for all wave speeds $c\geq c_{\min}=2$, but not for any wave speed less than $2$. Kolmogorov, Petrovskii and Piskunov \cite{Kolmogorov1937} established that, starting a solution $u$ of \eqref{eq:FKPP equation} from a Heaveside step function, there exists $\sigma(t)=2t+o(t)$ such that $u(x+\sigma(t),t)$ converges to $w_{c_{\min}}(x)$. Bramson \cite{Bramson1978,Bramson1983} refined the speed $\sigma$ to $\sigma(t)=2t-\frac{3}{2}\log t+\mathcal{O}(1)$, and showed that the domain of attraction is given by initial conditions $u_0$ such that $\liminf_{x\ra -\infty}\int_{x-H}^xu_0(y)dy>0$ for some $H<\infty$, and 
\[
\limsup_{x\ra\infty}\frac{1}{x}\log \Big[\int_x^{x(1+h)}u_0(y)dy\Big]\leq -1.
\]
Therefore we have convergence to the travelling wave with minimal wave speed when the initial condition has sufficiently light tails. This is a \textit{macroscopic selection principle}. The parallels with the phenomenon of convergence to the minimal QSD whenever the initial condition is sufficiently light, as in Theorem \ref{theo:domain of attraction of minimal QSD}, are clear.

We contrast this with a \textit{microscopic selection principle}, in which the introduction of a microscopic noise term has the effect of restricting the possible travelling waves to only the minimal one. This is natural, since any physical system must have a small amount of noise. The first such microscopic selection principle is due to Bramson et al. \cite{Bramson1986} in 1986. They considered a system parametrised by a parameter $\gamma<\infty$ (large $\gamma$ representing small noise) which has a hydrodynamic limit given by a reaction-diffusion equation as $\gamma\ra\infty$. They showed that for all $\gamma<\infty$ this system, seen from its rightmost particle, has a unique invariant distribution. Then they showed that the velocity of this stationary distribution, appropriately rescaled, converges to the minimal wave speed of the corresponding reaction-diffusion equation. We note, however, that this is a \textit{weak selection principle}, meaning that they established convergence of the wave speed but not of the profile of the stationary distribution. 

Later, Brunet and Derrida et al. initiated the study of the effect of noise on front propagation in \cite{Brunet1997,Brunet1999,Brunet2001,Brunet2006,Brunet2007}. One can incorporate noise either by considering a stochastic PDE, or an interacting particle system. In the former case, one may consider the stochastic FKPP,
\begin{equation}\label{eq:stochastic FKPP equation}
\frac{\partial u}{\partial t}=\frac{\partial^2 u}{\partial x^2}+u(1-u)+\epsilon \sqrt{u(1-u)}\dot{W},
\end{equation}
for small $\epsilon>0$. Brunet and Derrida conjectured that the speed of its travelling wave is given by 
\begin{equation}\label{eq:speedup of stochastic FKPP wave speed}
2-\pi^2\lvert \log \epsilon^2\rvert^{-2}+\calO((\log \lvert \log \epsilon\rvert)\lvert \log \epsilon\rvert^{-3}).
\end{equation}
This was proven by Mueller, Mytnik and Quastel \cite{Mueller2011}. We see that the wave speed converges to the minimal wave speed $v_{\min}=2$ as $\epsilon\ra 0$, but with a strikingly large correction term. 

Brunet and Derrida introduced a particle system of fixed size $N$, in which particles undergo repeated steps of branching and selection \cite{Brunet1997,Brunet1999}. Whereas various variants of this particle system have been considered in the literature, we descibe explicitly the $N$-branching Brownian motion ($N$-BBM), considered by Maillard in \cite{Maillard2016}. In this particle system, $N$ Brownian motions evolve in $\Rm$. Each particle branches at rate $1$, at which time the minimal particle is killed, keeping the mass constant. We may equivalently view the minimal particle as being killed at rate $N-1$, at which time it jumps onto the location of another particle chosen uniformly at random. We may view this particle system from a moving frame of reference in which its leftmost particle is fixed at $0$. We see that if the original particle system is moving to the right at velocity $c$ in the original frame of reference, then the particles have a drift $-c$ in the moving frame of reference (roughly speaking). In the moving frame of reference the particles are killed at rate $N-1$ (so effectively instantaneously) when they reach $0$, at which time they are redistributed to one of the other particles chosen uniformly at random. The connection with the Fleming-Viot particle system with drift $-1$ is clear. The difference is that in the Fleming-Viot process with drift $-1$ the drift is held constant while the killing rate varies, whilst in the $N$-BBM the killing rate is held constant while the drift varies in time.

De Masi, Ferrari, Presutti and Soprano-Loto established a hydrodynamic limit theorem for the $N$-BBM in \cite{DeMasi2019}. The hydrodynamic limit they established is given by solutions of a free boundary problem, $(\gamma_t,u_t)$, where $u_t$ is supported on $\{x:x\geq \gamma_t\}$. This has an infinite family of travelling wave solutions. These travelling wave solutions correspond to the QSDs of Brownian motion with constant drift $-1$ according to the following \cite[p.547]{DeMasi2019}. We recall that Brownian motion with constant drift $-1$ has a one-parameter family of QSDs $(\pi_{\lambda})_{0< \lambda\leq \frac{1}{2}}$ with associated eigenvalues $\lambda(\pi_{\lambda})=\lambda$. The minimal QSD is $\pi_{\frac{1}{2}}$, with eigenvalue $\lambda_{\min}=\frac{1}{2}$. We write $\pi_{\lambda}(x)$ for the density of $\pi_{\lambda}(dx)$. Then the aforementioned travelling wave solutions are given by
\begin{equation}\label{eq:travelling waves corresponding to QSDs}
u_c(t,x+ct)=w_c(x)=Z_c\pi_{\frac{1}{c^2}}\big(cx\big)\quad\text{for wave speeds}\quad \frac{1}{\sqrt{\lambda_{\min}}}=c_{\min}=\sqrt{2}\leq c<\infty,
\end{equation}
whereby $Z_c$ for $c\geq \sqrt{2}$ are normalisation constants. Therefore the profiles of the travelling waves are given by the profiles of the QSDs of $(X_t)_{0\leq t<\tau_{\partial}}$ given by \eqref{eq:QSDs of drift -1}, with space rescaled. We note in particular that the travelling wave with minimal wave speed corresponds to the minimal QSD for $(X_t)_{0\leq t<\tau_{\partial}}$. Convergence of the empirical measure of the $N$-BBM in its stationary distribution to this minimal travelling wave as $N\ra\infty$, as is proven for the Fleming-Viot process with drift $-1$ in Theorem \ref{theo:main theorem}, has been conjectured by Maillard (\cite[p.19]{Maillard2012}, \cite[p.1066]{Maillard2016}), Groisman and Jonckeere \cite[Conjecture 3.1]{Groisman2013}, N. Berestycki and Zhao \cite[p.659]{Berestycki2018}, and De Masi, Ferrari, Presutti and Soprano-Loto \cite[p.548]{DeMasi2019}. This remains an open problem.

Brunet and Derrida made a prediction similar to \eqref{eq:speedup of stochastic FKPP wave speed} for the particle systems they introduced, which was proven for a similar particle system by Bérard and Gouéré in \cite{Berard2010}. In particular, Bérard and Gouéré showed that the $N$-particle system they consider has an asymptotic speed $v_N$ given by $v_N=v_{\infty}-\frac{\chi}{\log^2N}+o(\log^{-2}N)$, for constants $v_{\infty},\chi>0$. In \cite{Durrett2011}, Durrett and Reminik considered another variant of the Brunet-Derrida particle system. They showed that as $N\ra\infty$ over fixed times, the particle system converges to solutions of a free boundary problem. Moreover they established that the process viewed from its tip converges to a unique stationary distribution as $t\ra\infty$ for fixed $N$, and that the asymptotic speed of the $N$-particle system converges to the minimal speed of travelling waves of the free boundary problem. Maillard has established in \cite[Theorem 1.1]{Maillard2016} that the $N$-BBM travels as a spectrally positive Levy process over a $\log^3N$ timescale, so random fluctuations in its speed occur over a $\log^3N$ timescale. We note that this assumes that the $N$-BBM is given the right initial distribution, so doesn't say anything about the behaviour for arbitrarily large $t$. N. Berestycki and Zhao \cite{Berestycki2018} considered a multidimensional generalisation of the $N$-BBM in \cite{Berestycki2018}. They showed that this particle system has an asymptotic speed in a possibly random direction, and that the asymptotic speed of the $N$-particle system converges to the speed of the corresponding minimal travelling wave, with a $\frac{-\pi^2}{\sqrt{2} \log^2N}$ correction. Groisman, Jonckeere and Mart{\'{i}}nez \cite{Groisman2020} considered a system whereby, for any pair of particles, at constant rate the larger of the two branches into two particles and the smaller of the two is killed. They established a hydrodynamic limit given by the FKPP equation, convergence as $t\ra\infty$ to a unique stationary distribution for fixed $N$, and that the asymptotic speed of the $N$-particle system converges as $N\ra\infty$ to the minimal travelling wave speed of the FKPP equation.

These provide for weak selection principles, in which only the asymptotic velocity of the $N$-particle system is shown to converge to the velocity of the minimal travelling wave, with an explicit correction term in some cases. To the authors' knowledge a full microscopic selection principle - in which one also shows that the empirical measure of the stationary $N$-particle system converges to the profile of the minimal travelling wave - has not been established in the travelling wave context. 

We recall that the selection principles established in \cite{Villemonais2015}, \cite{Asselah2016} and \cite{Champagnat2021a} proceeded by way of Lyapunov functions, an approach we shouldn't expect to work in the case of constant negative drift by \cite[Remark 3.3]{Villemonais2015}. However in the travelling wave context, the particle system should naturally have a drift $-c$ (roughly speaking) in the moving frame of reference, so we shouldn't expect a Lyapunov function approach to work. The selection principle we shall establish for the Fleming-Viot process with drift $-1$ is indeed proven without the use of Lyapunov functions. In the opinion of the present author, this proof is more likely to be applicable in the travelling wave context.

A full selection principle has been established for the Brownian bees particle system by J. Berestycki, Brunet, Nolen and Penington \cite{Berestycki2022}. This is a variant of the $N$-BBM whereby, instead of killing the leftmost particle at each selection step, one instead kills the particle furthest away from $0$. This has the effect of constraining the particles to a compact set, in contrast to the travelling wave context. They established that as $N\ra\infty$, this particle system approximates the solution of a free boundary problem, which we denote as $(R_t,u_t)$. The boundary of this free boundary problem is given by a ball of radius $R_t$ (dependent upon time). On the interior of the ball $u$ satisfies $\partial_tu=\Delta u+u$ with Dirichlet boundary conditions. The radius $R_t$ is uniquely chosen to conserve mass. As $t\ra \infty$, $u_t$ converges to the principal Dirichlet eigenfunction on a ball of radius $R_{\infty}$, with $R_{\infty}$ uniquely chosen so that the eigenvalue is $-1$. This is the unique steady state, giving the only possible long-time limit - there's no analogue of non-minimal travelling waves. They establish that this particle system converges to a unique stationary distribution for fixed $N$, and that this stationary distributions converges to the aforedescribed principal Dirichlet eigenfunction on the ball of radius $R_{\infty}$ as $N\ra\infty$.

\subsection{Structure of the paper}

In the following section we will prove Theorem \ref{theo:main theorem}. This is then followed by the appendix.

\section{Proof of Theorem \ref{theo:main theorem}}

We assume throughout that $N\geq 12$, so that $\lambda_N<\infty$ by Theorem \ref{theo:ergodicity for fixed N}. This incurs no loss of generality since Theorem \ref{theo:main theorem} is a statement about the limiting behaviour as $N\ra\infty$.

We begin with the following crucial lemma.
\begin{lem}\label{lem:liminf of lambda N}
We define $\iota_N:=-N\ln(1-\frac{1}{N})$ for $N\geq 12$. Then we have that $\lambda_N\geq \frac{\lambda_{\min}}{\iota_N} $ for all $N\geq 12$, so that $\liminf_{N\ra\infty}\lambda_N\geq \lambda_{\min}$ in particular.
\end{lem}

\begin{proof}[Proof of Lemma \ref{lem:liminf of lambda N}]
We fix $N\geq 12$. We consider a \textit{fixed deterministic} initial condition $\vec{X}_0^N\in \Rm_{>0}^N$ for the $N$-particle system, with empirical measure $m_0^N=\Theta^N(\vec{X}^N_0)$. We define the mean measure
\[
\rho_t(\cdot):=\expE_{\vec{X}_0^N}\big[\big(1-\frac{1}{N}\big)^{NJ^N_t}m^N_t(\cdot)\big].
\]
We now show that
\begin{equation}\label{eq:rho equals prob of survival}
\Pm_{m_0^N}(X_t\in \cdot,\tau_{\partial}>t)=\rho_t(\cdot)\quad \text{for all}\quad t\geq 0.
\end{equation}
\begin{proof}[Proof of \eqref{eq:rho equals prob of survival}]
Our strategy is to apply a PDE uniqueness theorem. We take arbitrary $K >\max_{1\leq i\leq N}X^{N,i}_0$. We sort particles into blue and red particles as follows. Initially all particles are blue. If a blue particle hits $K$ or is killed and jumps onto a red particle, it becomes red. If a red particle is killed and jumps onto a blue particle, it becomes blue. Particles don't otherwise change colour. The set of indices of blue particles at time $t$ is $\mathscr{B}^K_t$. We may therefore define
\[
m_t^{N,K}(\cdot):=\frac{1}{N}\sum_{i=1}^N\Ind(i\in \mathscr{B}^K_t)\delta_{X^{N,i}_t}(\cdot),\quad \rho^K_t(\cdot):=\expE_{\vec{X}_0^N}\big[\big(1-\frac{1}{N}\big)^{NJ^N_t}m^{N,K}_t(\cdot)\big].
\]
We define $\mathcal{S}_K:=\{\varphi\in C_c^{\infty}([0,\infty)\times [0,K]): \varphi\equiv 0\text{ on }[0,\infty)\times \{0,K\}\}$. It is straightforward to see that
\[
\big(1-\frac{1}{N}\big)^{NJ^N_t}\langle m_t^{N,K},\varphi\rangle - \langle m_0^{N,K},\varphi\rangle-\int_0^t\big(1-\frac{1}{N}\big)^{NJ^N_s}\langle m_s^{N,K},(L+\frac{\partial}{\partial s})\varphi\rangle ds \quad  \text{is a martingale}
\]
for all $\varphi\in \mathcal{S}_K$. It follows that $\rho_t^K$ is a solution of 
\begin{equation}\label{eq:weak solution of PDE}
u_t(\varphi)-u_0(\varphi)-\int_0^t\langle u_s,(L+\frac{\partial}{\partial s})\varphi\rangle ds=0 \quad \text{for all $\varphi\in \mathcal{S}_K$.}
\end{equation}
It follows from parabolic regularity that $\rho^K_t$ has a $C^{\infty}((0,\infty)\times (0,K))$ density, so $\rho_t\in L^1(\Leb)$ with $\lvert\lvert\rho_t\rvert\rvert_{L^1(\Leb)}\leq 1$ for all $t>0$ (since it has a mass at least $1$ at all times). We now define the stopping time $\tau_K:=\inf\{t>0:X_{t-}\geq K\}$ and observe that $\Pm_{\mu}(X_t\in \cdot,\tau_{\partial}\wedge \tau_K>t)$ is also a solution of \eqref{eq:weak solution of PDE}, for any $\mu\in \calP((0,K))$. We have a uniqueness theorem for solutions of \eqref{eq:weak solution of PDE} with $L^1$ initial data, \cite[Corollary 3.5]{Porretta2015a}, from which we conclude that 
\[
\rho^K_t(\cdot)=\rho^K_{\epsilon}(1)\Pm_{\frac{\rho^K_{\epsilon} }{\rho^K_{\epsilon}(1)}}(X_{t-\epsilon}\in \cdot,\tau_{\partial}\wedge \tau_K> t-\epsilon )\quad\text{for all}\quad 0<\epsilon<t. 
\]
Taking the limit as $\epsilon\ra 0$ with fixed $t>0$, we see that 
\[
\Pm_{m_0^N}(X_t\in \cdot,\tau_{\partial}\wedge \tau_K>t)=\rho_t^{K}(\cdot)\quad \text{for all}\quad t\geq 0.
\]
We therefore obtain \eqref{eq:rho equals prob of survival} by taking $K\ra \infty$.
\end{proof}

Evaluating \eqref{eq:rho equals prob of survival} against $1$, we see that
\begin{equation}\label{eq:liminf lambda N lemma mass of rhot expression}
\rho_t(1)=\expE_{\vec{X}_0^N}\big[e^{-\iota_N J^N_t}\big]=\Pm_{m_0^N}(\tau_{\partial}>t),\quad t\geq 0.
\end{equation}
Note that in the above expression on the right hand side, $\Pm_{m_0^N}(\tau_{\partial}>t)$ should be understood to be the probability of survival at time $t$ if we start a single particle with initial distribution given by the fixed empirical measure $m_0^N$.

Since $m_0^N$ is compactly supported, $\Law_{m_0^N}(X_t\lvert \tau_{\partial}>t)\ra \pi_{\min}$ in total variation as $t\ra\infty$ by Theorem \ref{theo:domain of attraction of minimal QSD} (in fact we only require that $\pi_{\min}$ is a Yaglom limit), so that $\Pm_{m^N_0}(\tau_{\partial}>t+1\lvert \tau_{\partial}>t)\ra \Pm_{\pi_{\min}}(\tau_{\partial}>1)=e^{-\lambda_{\min}}$ as $t\ra\infty$. Therefore there exists $c_0=c_0(m_0^N,\lambda)<\infty$ for all $\lambda<\lambda_{\min}$ such that
\begin{equation}\label{eq:liminf lambda N lemma c0 bounded}
c_0(m_0^N,\lambda):=\sup_{t\geq 0}[ e^{\lambda t}\Pm_{m_0^N}(\tau_{\partial}>t)]<\infty.
\end{equation}
Combining \eqref{eq:liminf lambda N lemma mass of rhot expression} with \eqref{eq:liminf lambda N lemma c0 bounded}, we see that
\[
\begin{split}
\Pm_{\vec{X}_0^N}(\iota_N\frac{J_t^N}{t}<\lambda-\epsilon)=\Pm_{\vec{X}_0^N}(\iota_NJ_t^N <\lambda t-\epsilon t)=\Pm_{\vec{X}_0^N}(\epsilon t<\lambda t-\iota_N J_t^N)\\\leq e^{-\epsilon t}\expE_{\vec{X}_0^N}[e^{\lambda t-\iota_NJ_t^N}]
=e^{-\epsilon t}[e^{\lambda t}\Pm_{m_0^N}(\tau_{\partial}>t)]\leq c_0e^{-\epsilon t}.
\end{split}
\]
It follows from the Borel-Cantelli lemma that 
\[
\Pm_{\vec{X}_0^N}(\iota_N\frac{J^N_t}{t}<\lambda-\epsilon\; \text{infinitely often at integer times $t\in\Nm$})=0.
\]
Since $\lambda<\lambda_{\min}$ and $\epsilon>0$ are arbitrary, we see that $\liminf_{t\ra\infty}\frac{J^N_t}{t}\geq \frac{\lambda_{\min}}{\iota_N}$ almost surely, for any fixed initial distribution $\vec{X}_0^N$. Since this holds for arbitrary fixed deterministic initial condition, it holds for arbitrary initial distribution, whence Lemma \ref{lem:liminf of lambda N} follows 
\end{proof}

\begin{defin}
We consider the discrete time Markov process obtained by observing the $N$-particle Fleming-Viot process at successive jump times, $(\vec{X}_{\tau_n^N}^N)_{0\leq n<\infty}$, which we refer to as the $N$-particle jump-time process, and denote as $(\vec{Y}_n)_{0\leq n<\infty}:=(\vec{X}_{\tau_n^N}^N)_{0\leq n<\infty}$.
\end{defin}

\begin{prop}\label{prop:exist and uniq of jump-time process stationary distn}
There exists a unique stationary distribution for the $N$-particle jump-time process, which we denote as $\phi^N\in \calP(\Rm_{>0}^N)$, for all $N\geq 12$. This is related to the stationary distribution $\psi^N\in \calP(\Rm_{>0}^N)$ of the $N$-particle Fleming-Viot process by
\begin{equation}\label{eq:formula for jump-time stationary distribution}
\phi^N(\cdot)=\frac{1}{N\lambda_N}\expE_{\psi^N}\Big[\sum_{0< \tau_n\leq 1}\delta_{\vec{X}^N_{\tau_n}}(\cdot)\Big].
\end{equation}
\end{prop}

\begin{proof}[Proof of Proposition \ref{prop:exist and uniq of jump-time process stationary distn}]
We define $\phi^N(\cdot)$ by \eqref{eq:formula for jump-time stationary distribution}. This is a well-defined probability measure by Theorem \ref{theo:ergodicity for fixed N}. Our goal is to show that it is the unique stationary distribution for the jump-time process.

We write $\bfQ$ for the transition kernel of the jump-time process. We observe that for all $t\in \Nm_{>0}$ we have
\[
\phi^N(\cdot)=\frac{1}{N\lambda_N t}\expE_{\psi^N}\Big[\sum_{0< \tau_n\leq t}\delta_{\vec{X}^N_{\tau_n}}(\cdot)\Big],
\]
since $\psi^N$ is the stationary distribution of $(\vec{X}^N_t)_{t\geq 0}$. We fix $t\in\Nm_{>0}$ and observe that
\[
\begin{split}
\lambda_Nt\phi^N \bfQ(\cdot)=\sum_{n\geq 0}\expE_{\psi^N}[\Ind(\tau_n\leq t)\bfQ(\vec{X}^N_{\tau_n},\cdot)]\\
=\sum_{n\geq 0}\expE_{\psi^N}[\Ind(\tau_n\leq t)\delta_{\vec{X}^N_{\tau_{n+1}}}(\cdot)]=\sum_{n\geq 1}\expE_{\psi^N}[\Ind(\tau_{n-1}\leq t)\delta_{\vec{X}^N_{\tau_{n}}}(\cdot)].
\end{split}
\]
We observe that
\[
\lvert\lvert\lambda_Nt\phi^N \bfQ(\cdot)-\lambda_nt\phi^N\rvert\rvert_{\TV}\leq 2.
\]
It follows that
\[
\lvert\lvert\phi^N \bfQ(\cdot)-\phi^N(\cdot)\rvert\rvert_{\TV}\leq \frac{2}{\lambda_Nt}.
\]
Since $t\in\Nm_{>0}$ is arbitrary, it follows that $\phi^N$ is a stationary distribution for the jump-time process.

The uniqueness of $\phi^N(\cdot)$ follows from Theorem \ref{theo:ergodicity for fixed N} and Birkhoff's theorem.
\end{proof}

We recall that $\chi^N:=\Theta^N_{\#}\psi^N$ is the stationary empirical measure of the $N$-particle Fleming-Viot process. We similarly define the stationary empirical measure of the $N$-particle jump-time process by 
\begin{equation}
\Upsilon^N:=\Theta^N_{\#}\phi^N\in \calP(\calP(\Rm_{>0})).
\end{equation}
We then define the mean measures of these two stationary empirical measures,
\begin{equation}
\xi^N(\cdot):=\expE[\chi^N(\cdot)]=\expE_{\vec{X}^N\sim \psi^N}[(\Theta^N_{\#}\vec{X}^N)(\cdot)],\quad \varpi^N(\cdot):=\expE[\Upsilon^N(\cdot)]=\expE_{\vec{X}^N\sim \phi^N}[(\Theta^N_{\#}\vec{X}^N)(\cdot)].
\end{equation}

We define $\calM_{\geq 0}(\Rm_{>0})$ and $\calB_{\geq 0}(\Rm_{\geq 0})$ to respectively be the space of non-negative (not necessarily finite) Borel measures on $\Rm_{>0}$ and non-negative (not necessarily finite) Borel functions on $\Rm_{>0}$. We then define the Green kernel
\[
Gf(x):=\expE_x\Big[\int_0^{\tau_{\partial}}f(X_s)ds\Big],\quad f\in \calB_{\geq 0}(\Rm_{\geq 0}),\quad \mu G(\cdot):=\expE_{\mu}\Big[\int_0^{\tau_{\partial}}\delta_{X_s}(\cdot) ds\Big],\quad \mu\in\calM_{\geq 0}(\Rm_{>0}).
\]
Note that $G$ is not bounded, in particular $G1(x)\ra \infty$ as $x\ra\infty$, which necessitates the above restriction to non-negative measures and Borel functions. The following proposition relates $\xi^N$ and $\varpi^N$.
\begin{prop}\label{prop:relation between piN and varpiN}
The mean stationary measures are related by $\xi^N =\lambda_N\varpi^N G $.
\end{prop}

Before proving Proposition \ref{prop:relation between piN and varpiN}, we firstly demonstrate how it provides for the tightness of $\{\varpi^N:N\geq 12\}$ in $\calP(\Rm_{\geq} 0)$. It follows from Proposition \ref{prop:relation between piN and varpiN} that
\begin{equation}\label{eq:G1 integral of varpiN}
\lambda_N\varpi^N(G1)=\xi^N(1)=1\quad \text{for all}\quad N\geq 12.
\end{equation}
Since $G1(x)\ra \infty$ as $x\ra \infty$ and $\liminf_{N\ra\infty}\lambda_N\geq \lambda_{\min}>0$, it follows from \eqref{eq:G1 integral of varpiN} that $\{\varpi_N:N\geq 12\}$ is tight in $\calP(\Rm_{\geq 0})$. It follows in particular that $\{\Upsilon^N:N\geq 12\}$ is tight in $\calP(\calP(\Rm_{\geq 0}))$.

Establishing any kind of tightness appears to be very difficult to accomplish by way of estimates. Thus, whilst seemingly almost trivial, this observation is crucial to the proof. 

We offer the following heuristic interpretation of \eqref{eq:G1 integral of varpiN}. There are two ways of calculating the total mass of all particles up to a large time $t$. One is simply to multiply $N$ by $t$. The other is to count the mass contributed by each particle lifetime, given by the time that particle survives after birth. The two must agree. If a particle is born far away from $0$, we should expect it to survive a long time, hence to contribute a lot of mass. However, since $\inf_N\lambda_N>0$, if $\varpi^N$ were to apply too much mass far away from $0$, then we would have too many particle births occurring far away from $0$, hence contributing too much mass.

\begin{proof}[Proof of Proposition \ref{prop:relation between piN and varpiN}]
The process $(X_t)_{0\leq t<\tau_{\partial}}$ is a $C_0$-Feller process on $\Rm_{>0}$, meaning that the submarkovian transition semigroup $(P_t)_{t\geq 0}$ defined by
\[
P_t:C_0(\Rm_{>0})\ni f\mapsto (\Rm_{>0}\ni x\mapsto \expE_x[f(X_t)\Ind(\tau_{\partial}>t)])\in C_0(\Rm_{>0})
\]
is a well-defined $C_0$-Feller semigroup. We denote its infinitesimal generator as $L$. We write $\calD(L)$ for the domain of $L$ and $\calD_{\geq 0}(L)$ for the non-negative elements of $\calD(L)$.

We consider the $N$-particle Fleming-Viot process with initial condition $\vec{X}^N_0\sim \psi^N$. We claim that
\begin{equation}\label{eq:martingale for mean stationary measures relationship}
\langle m_t^N,\varphi\rangle -\langle m_0^N,\varphi\rangle-\int_0^t\langle m_s^N,L\varphi\rangle ds-\frac{1}{N}\sum_{\tau_n\leq t}\langle m_{\tau_n}^N,\varphi\rangle\quad\text{is a martingale for all $\varphi\in \calD_{\geq 0}(L)$.}
\end{equation}
We fix $\varphi\in \calD_{\geq 0}(L)$. It is clear that \eqref{eq:martingale for mean stationary measures relationship} is a local martingale, localised up to the times $\tau_n$. We now fix $t<\infty$. Then we have that
\[
\expE[\frac{1}{N}\sum_{\substack{\tau_n\leq t\\ n\leq k}}\langle m_{\tau_n}^N,\varphi\rangle]\leq 2\lvert \lvert \varphi\rvert\rvert_{\infty}+t\lvert\lvert L\varphi\rvert\rvert_{\infty}\quad\text{for all}\quad k<\infty.
\]
By the monotone convergence theorem we see that 
\[
\expE[\frac{1}{N}\sum_{\substack{\tau_n\leq t\\ n> k}}\langle m_{\tau_n}^N,\varphi\rangle] \ra 0\quad\text{as}\quad k\ra\infty.
\]
This suffices to show that \eqref{eq:martingale for mean stationary measures relationship} is a martingale for all $\varphi\in \calD_{\geq 0}(L)$.

It follows from Proposition \ref{prop:exist and uniq of jump-time process stationary distn} that
\begin{equation}\label{eq:expectation of sum of jumps stationary ics}
\expE_{\vec{X}^N_0\sim \psi^N}\Big[\frac{1}{N}\sum_{\tau_n\leq 1}\langle m_{\tau_n}^N,\varphi\rangle\Big]=\lambda_N\varpi^N(\varphi).
\end{equation}

Taking the expectation of \eqref{eq:martingale for mean stationary measures relationship}, using that $\vec{X}_0^N\sim \psi^N$, and applying \eqref{eq:expectation of sum of jumps stationary ics}, we see that
\begin{equation}\label{eq:piN omega N relation by L}
\xi^N(L\varphi)=-\lambda_N \varpi^N(\varphi)\quad\text{for all}\quad \varphi\in \calD_{\geq 0}(L).
\end{equation}

We define $G_t:=\int_0^tP_sds$, which we note is bounded, for all $0\leq t<\infty$. We have that
\begin{equation}\label{eq:Green kernel Gt calculation}
(P_h-P_0)G_t=G_t(P_h-P_0)=\int_h^{t+h}P_sds-\int_0^tP_s ds=\int_t^{t+h}P_sds- \int_0^hP_s ds=(P_t-P_0)G_h.
\end{equation}
It follows from \eqref{eq:Green kernel Gt calculation} that $G_tf\in \calD_{\geq 0}(L)$ for all $t\geq 0$ and $f \in C_0(\Rm_{>0};\Rm_{\geq 0})$, and that
\[
(P_t-P_0)f=\lim_{h\ra 0}(P_t-P_0)\frac{G_hf}{h}=\lim_{h\ra 0}\frac{(P_h-P_0)}{h}G_tf=L(G_tf)\quad\text{for all}\quad f\in C_0(\Rm_{>0};\Rm_{\geq 0}).
\]
It then follows from \eqref{eq:piN omega N relation by L} that
\[
\xi^N(f-P_tf)=-\xi^N((P_t-P_0)
f)=-\xi^N(LG_tf)=\lambda_N\varpi^N(G_tf)\quad\text{for all}\quad f\in C_0(\Rm_{>0};\Rm_{\geq 0}).
\]
Taking the monotone limit as $t\ra\infty$, we see that $\lambda_N\varpi^N(Gf)=\xi^N(f)\leq \lvert\lvert f\rvert\rvert_{\infty}$ for all $f\in C_0(\Rm_{>0};\Rm_{\geq 0})$. It follows from the monotone convergence theorem that $(\lambda_N\varpi^NG)(1)=\xi^N(1)$. Since $\xi^N$ is a probability measure, it follows that $\lambda_N\varpi^NG$ is also a probability measure, which satisfies $\xi^N(f)=(\lambda_N\varpi^NG)(f)$ for all $f\in C_0(\Rm_{>0};\Rm)$. We conclude the proof of Proposition \ref{prop:relation between piN and varpiN} by applying the Riesz-Markov-Kakutani representation theorem.

Alternatively, we could have proven Proposition \ref{prop:relation between piN and varpiN} by way of a folklore proposition, which is stated and proven in \cite[Proposition 3.1]{Benaim2023}.
\end{proof}

Having established the tightness of $\{\Upsilon^N\}$ in $\calP(\calP(\Rm_{\geq 0 }))$ and of $\{\varpi^N\}$ in $\calP(\Rm_{\geq 0 })$, we would like to establish tightness in $\calP(\calP(\Rm_{> 0 }))$ and $\calP(\Rm_{> 0 })$ respectively, which requires preventing mass from accumulating at $0$. We accomplish this with the following proposition.

\begin{prop}\label{prop:subsequential limits apply mass to R+ only}
We have that $\{\Upsilon^N:N\geq 12\}$ is tight in $\calP(\calP(\Rm_{>0}))$ and $\{\varpi^N:N\geq 12\}$ is tight in $\calP(\Rm_{>0})$.
\end{prop}

\begin{proof}[Proof of Proposition \ref{prop:subsequential limits apply mass to R+ only}]
For all $\epsilon,\delta>0$ we have from Proposition \ref{prop:exist and uniq of jump-time process stationary distn} that
\[
\expE_{\vec{X}^N_0\sim \psi^N}\Big[\frac{1}{N}\sum_{\tau_n\leq 1}\Ind[m^N_{\tau_n}(B(0,\delta))>\epsilon]\Big]=\lambda_N\Upsilon^N(\{m:m(B(0,\delta))>\epsilon\}).
\]
It therefore suffices to establish the following lemma.
\begin{lem}\label{lem:bound on number of jumps while mass all at the boundary}
For all $\epsilon>0$ there exists $\delta>0$ such that 
\begin{equation}
\expE_{\vec{X}^N_0\sim \psi^N}\Big[\frac{1}{N}\sum_{\tau_n\leq 1}\Ind(m^N_{\tau_n}(B(0,\delta))>\epsilon)\Big]\ra 0\quad\text{as}\quad N\ra\infty.
\end{equation}
\end{lem}
The proof of this lemma requires a crude estimate. We defer its proof to the appendix.

This concludes the proof of Proposition \ref{prop:subsequential limits apply mass to R+ only}.
\end{proof}

For $\mu\in \calP(\Rm_{> 0})$ and $y>0$ we define
\begin{equation}\label{eq:definition of Ty}
T_y(\mu):=\inf\{t>0:-\ln \Pm_{\mu}(\tau_{\partial}>t)>y\}.
\end{equation}
We then define the following flows on $\calP(\Rm_{>0})$,
\begin{equation}
\begin{split}
\theta_t:\calP(\Rm_{>0})\ni \mu\mapsto \Law_{\mu}(X_t\lvert \tau_{\partial}>t)\in \calP(\Rm_{>0}),\quad t\geq 0,\\
\vartheta_y:\calP(\Rm_{>0})\ni \mu\mapsto \Law_{\mu}(X_{T_y(\mu)}\lvert \tau_{\partial}>T_y(\mu))\in \calP(\Rm_{>0}),\quad y\geq 0.
\end{split}
\end{equation}
It is well-known and easy to check that $\theta_t\circ \theta_s=\theta_{t+s}$ for $t,s\geq 0$. Similarly we have 
\begin{equation}\label{eq:vartheta flow}
\vartheta_y\circ \vartheta_z=\vartheta_{y+z}\quad\text{and}\quad T_{y+z}(\mu)=T_z(\vartheta_y(\mu))+T_y(\mu) \quad\text{for}\quad y,z\geq 0 \quad\text{and}\quad\mu\in\calP(\Rm_{>0}).
\end{equation}
\begin{proof}[Proof of \eqref{eq:vartheta flow}]
We fix arbitrary $y,z\geq 0$ and $\mu\in\calP(\Rm_{>0})$. We define $t_0:=T_y(\mu)$ and $t_1:=T_z(\vartheta_y(\mu))$. Then we have that
\[
\Pm_{\mu}(\tau_{\partial}>t_0+t_1)=\Pm_{\Law_{\mu}(X_{t_0}\lvert \tau_{\partial}>t_0)}(\tau_{\partial}>t_1)\Pm_{\mu}(\tau_{\partial}>t_0)=\Pm_{\vartheta_y(\mu)}(\tau_{\partial}>t_1)e^{-y}=e^{-(z+y)}.
\]
It follows that $T_{y+z}(\mu)=T_y(\mu)+T_z(\vartheta_y(\mu))$. We then have that 
\[
\vartheta_{y+z}(\mu)=\theta_{T_{y+z}(\mu)}(\mu)=\theta_{t_0+t_1}(\mu)=\theta_{t_1}\circ\theta_{t_0}(\mu)=\theta_{T_z(\vartheta_y(\mu))}\circ \vartheta_y(\mu)=\vartheta_z\circ \vartheta_y(\mu).
\]
\end{proof}
We further define
\begin{equation}
T^N_y:=\inf\{t>0:J^N_t>y\}=\tau^N_{\lfloor Ny+1\rfloor},\quad 0\leq y<\infty.
\end{equation}

The following characterises subsequential limits of stationary empirical measures as invariant measures of these flows.
\begin{prop}\label{prop:subsequential limits are invariant measures}
We suppose that the stationary empirical measure of the $N$-particle Fleming-Viot process (respectively the $N$-particle jump-time process), $\chi^N\in \calP(\calP(\Rm_{>0}))$ (respectively $\Upsilon^N\in \calP(\calP(\Rm_{>0}))$), converges in $\calP(\calP(\Rm_{>0}))$ along a subsequence to $\chi\in \calP(\calP(\Rm_{>0}))$ (respectively $\Upsilon\in \calP(\calP(\Rm_{>0}))$). Then $\chi$ (respectively $\Upsilon$) is invariant under the flow $(\theta_t)_{t\geq 0}$ (respectively $(\vartheta_y)_{y\geq 0}$). 
\end{prop}

\begin{proof}[Proof of Proposition \ref{prop:subsequential limits are invariant measures}]
We  establish Proposition \ref{prop:subsequential limits are invariant measures} by applying a hydrodynamic limit theorem. Villemonais \cite[Theorem 2.2]{Villemonais2011} has established a hydrodynamic limit theorem for the Fleming-Viot process under general conditions. However, this theorem is not in the form we need here. Whilst we could adapt \cite[Theorem 2.2]{Villemonais2011} to put it into the requisite form, \cite[Theorem 2.10]{Tough2022} includes the present setting and is already in the requisite form. A statement of \cite[Theorem 2.10]{Tough2022}, Theorem \ref{theo:tight ic hydrodynamic limit theorem}, can be found in Appendix \ref{appendix:tight ic hydrodynamic limit theorem}.

We consider along our given subsequence Fleming-Viot proesses $(\vec{X}^N_t)_{t\geq 0}$ with stationary initial conditions $\vec{X}^N_0\sim \psi^N$. We write $m^N_t$ for the corresponding empirical measures $m^N_t:=\Theta^N(\vec{X}^N_t)$. We fix $t\geq 0$. We denote by $m_0$ a random variable with distribution $m_0\sim \chi$, so that $m^N_0$ converges in $\calP(\Rm_{>0})$ in distribution to $m_0\sim \chi$. It follows from Theorem \ref{theo:tight ic hydrodynamic limit theorem} that $m^N_t$ converges in $\calP(\Rm_{>0})$ in distribution to $\theta_t(m_0)$. Since $\vec{X}^N_0\sim \psi^N$, however, we have that $\vec{X}^N_t\sim \psi^N$ for all $N$, so that $m^N_t$ also converges in $\calP(\Rm_{>0})$ in distribution to a random variable with distribution $\chi$. Therefore $\Law(\theta_t(m_0))=\chi$, so that $(\theta_t)_{\#}\chi=\chi$. Since $t\geq 0$ is arbitrary, we conclude the proof that $\chi$ is $(\theta_t)_{t\geq 0}$-invariant.

The proof for subsequential limits of $\Upsilon^N$ follows in the same manner. We instead take $\vec{X}^N_0\sim \phi^N$ and fix $y\geq 0$. Again writing $m_0\sim \Upsilon$ for a random variable which is a limit in distribution of our initial empirical measure, Theorem \ref{theo:tight ic hydrodynamic limit theorem} ensures that $m^N_{T^N_y}$ converges in $\calP(\Rm_{>0})$ in distribution to $\vartheta_y(m_0)$. Since $T^N_y= \tau^N_{\lfloor Ny+1\rfloor}$ and $\phi^N$ is stationary for the jump-time process, we conclude as before that $\Upsilon=(\vartheta_y)_{\#}\Upsilon$, for arbitrary $y\geq 0$.
\end{proof}

We have established the tightness of $\{\Upsilon^N:N\geq 12\}$ in $\calP(\calP(\Rm_{>0}))$, and characterised subsequential limits as invariant measures of the flow $(\vartheta_y)_{y\geq 0}$. We will need more information to uniquely characterise subsequential limits, however. The requisite extra information is provided by the following proposition.
\begin{prop}\label{prop:killing rate subsequential limits of stopped process}
We suppose that the stationary empirical measure of the $N$-particle jump-time process, $\Upsilon^N\in \calP(\calP(\Rm_{>0}))$, converges in $\calP(\calP(\Rm_{>0}))$ along a subsequence to $\Upsilon\in \calP(\calP(\Rm_{>0}))$. Then $\Upsilon(T_1)=\int_{\calP(\Rm_{>0})}T_1(m)\Upsilon(dm)\leq T_1(\pi_{\min})$. Moreover if along this subsequence we have $\limsup_{N\ra\infty}\lambda_N>\lambda_{\min}$, then $\Upsilon(T_1)<T_1(\pi_{\min})$.
\end{prop}

\begin{proof}[Proposition \ref{prop:killing rate subsequential limits of stopped process}]
We fix $N\geq 12$ and establish that
\begin{equation}\label{eq:expected value of jump time from jump time stationary}
\expE_{\phi^N}[\tau_1^N]=\frac{1}{N\lambda_N}.
\end{equation}

We note that we have not yet proven that $\tau_1^N$ is integrable, which is included in the statement of \eqref{eq:expected value of jump time from jump time stationary}. It follows from Theorem \ref{theo:ergodicity for fixed N} that
\begin{equation}\label{eq:almost sure convergence of jump times method 1}
\frac{\tau_n^N}{n}\overset{a.s.}{\ra} \frac{1}{N\lambda_N}\quad\text{as}\quad n\ra\infty,\quad\text{for any initial condition.}
\end{equation}
On the other hand, the $N$-particle jump-time process along with the corresponding intervals between jump times, in its stationary distribution $\phi^N$, defines a measure-preserving system, which we denote as $(X,\mu,T)$. We denote the $\sigma$-algebra of $T$-invariant sets as $\mathscr{C}$. It follows from Birkhoff's theorem that
\begin{equation}\label{eq:almost sure convergence of jump times method 2}
\frac{1}{n}\sum_{0\leq k\leq n-1}[(\tau^N_{k+1}-\tau^N_k)\wedge C]\ra \expE[\tau^N_1\wedge C\lvert \mathscr{C}]\quad \mu-\text{almost surely as}\quad n\ra\infty,
\end{equation}
for any $C<\infty$. We note that since the convergence in \eqref{eq:almost sure convergence of jump times method 1} holds for any initial condition, it must hold $\mu$-almost surely. 

Since $\sum_{0\leq k\leq n-1}[(\tau^N_{k+1}-\tau^N_k)\wedge C]\leq \tau_n^N$ for any $C<\infty$, it follows from \eqref{eq:almost sure convergence of jump times method 1} and \eqref{eq:almost sure convergence of jump times method 2} that $\expE[\tau^N_1\wedge C]\leq  \frac{1}{N\lambda_N}$ for any $C<\infty$. It follows from the monotone convergence theorem that $\tau^N_1$ is integrable. We can therefore take $C=+\infty$ in \eqref{eq:almost sure convergence of jump times method 2}. Since $\sum_{0\leq k\leq n-1}[\tau^N_{k+1}-\tau^N_k]= \tau_n^N$ for all $n<\infty$, \eqref{eq:expected value of jump time from jump time stationary} follows from \eqref{eq:almost sure convergence of jump times method 1} and \eqref{eq:almost sure convergence of jump times method 2}

We can identify $T^N_1=\tau^N_{N+1}$, so that 
\[
\expE_{\phi^N}[T^N_1]=\frac{N+1}{N\lambda_N}\quad\text{for all}\quad N\geq 12.
\]

We now take along our given subsequence $N$-particle Fleming-Viot processes $(\vec{X}^N_t)_{t\geq 0}$ with initial conditions $\vec{X}^N_0\sim \phi^N$. It follows from Theorem \ref{theo:tight ic hydrodynamic limit theorem} that $T_1^N\overset{d}{\ra} T_1(m)$, whereby $m$ is a random variable with distribution $m\sim \Upsilon$. Applying Skorokhod's representation theorem, this becomes almost sure convergence on a new common probability space. It then follows by Fatou's lemma and Lemma \ref{lem:liminf of lambda N} that along this subsequence we have
\begin{equation}\label{eq:Fatou lemma usage}
\Upsilon(T_1)=\expE[T_1(m)]=\expE[\liminf_{N\ra\infty}T^N_1]\leq \liminf_{N\ra\infty}\expE[T^N_1]= \frac{1}{\limsup_{N\ra\infty}\lambda_N}\leq \frac{1}{\lambda_{\min}}=T_1(\pi_{\min}),
\end{equation}
with a strict inequality if $\limsup_{N\ra\infty}\lambda_N>\lambda_{\min}$.

This concludes the proof of Proposition \ref{prop:killing rate subsequential limits of stopped process}.
\end{proof}

We now consider such a subsequential limit $\Upsilon\in \calP(\calP(\Rm_{>0}))$ of $\{\Upsilon^N:N\geq 12\}$. Since $\Upsilon$ is $\vartheta_1$-invariant by Propositon \ref{prop:subsequential limits are invariant measures}, we have that
\[
\vartheta_1:\calP(\Rm_{>0})\ra \calP(\Rm_{>0})
\]
is a measure-preserving transformation on the probability space $(\calP(\Rm_{>0}),\Upsilon)$, with $T_1\in L^1(\Upsilon)$ by Proposition \ref{prop:killing rate subsequential limits of stopped process}. We write $\mathscr{C}$ for the $\sigma$-algebra of $\vartheta_1$-invariant sets. It follows from Birkhoff's theorem that
\begin{equation}\label{eq:almost sure convergence of Tn}
\frac{1}{n}T_n(m)\overset{\eqref{eq:vartheta flow}}{=}\frac{1}{n}\sum_{0\leq k<n}T_1(\vartheta_k(m))\ra \expE[T_1\lvert \mathscr{C}](m)\quad\text{as}\quad n\ra\infty,\quad \Upsilon\text{-almost surely.}
\end{equation}
We define $F:=\{m\in\calP(\Rm_{>0}):\lim_{n\ra\infty}\frac{1}{n}T_n(m)\leq T_1(\pi_{\min})\}\in\mathscr{C}$. If $m\in F$, then
\[
\limsup_{n\ra\infty}\frac{1}{n}\inf\{t>0:\Pm_{m}(\tau_{\partial}>t)<e^{-n}\}\leq T_1(\pi_{\min})=\frac{1}{\lambda_{\min}},
\] 
from which it follows that
\[
\limsup_{t\ra\infty}\frac{1}{t}\ln\Pm_{m}(\tau_{\partial}>t)\leq -\lambda_{\min}.
\]
It therefore follows from Theorem \ref{theo:domain of attraction of minimal QSD} that 
\begin{equation}
\vartheta_n(m)\overset{\TV}{\ra} \pi_{\min} \quad\text{as}\quad n\ra\infty \quad\text{for all} \quad m\in F. 
\end{equation}
Since $\Upsilon$ is $\vartheta_1$-invariant and $F\in \mathscr{C}$, it follows that $m=\pi_{\min}$ for $\Upsilon$-almost every $m\in F$. We see that $\expE[T_1\lvert \mathscr{C}]=T_1(\pi_{\min})$ on $F$. Using Proposition \ref{prop:killing rate subsequential limits of stopped process}, we therefore have that
\[
\begin{split}
T_1(\pi_{\min})\geq \expE[T_1]=\expE[\expE[T_1\lvert \mathscr{C}](m)\Ind(m\in F)]+\expE[\expE[T_1\lvert \mathscr{C}](m)\Ind(m\in F^c)]\\
=T_1(\pi_{\min})\Upsilon(F)+\expE[\expE[T_1\lvert \mathscr{C}](m)\Ind(m\in F^c)].
\end{split}
\]
Therefore 
\[
T_1(\pi_{\min})\Upsilon(F^c)=\expE[\expE[T_1\lvert \mathscr{C}](m)\Ind(m\in F^c)].
\]
On the other hand, by \eqref{eq:almost sure convergence of Tn} and the definition of $F$ we have that $\expE[T_1\lvert \mathscr{C}](m)>T_1(\pi_{\min})$ for $\Upsilon$-almost every $m\in F^c$, so that $\Upsilon(F^c)=0$. 

We conclude that $\Upsilon(\pi_{\min})=1$ so that $\Upsilon=\delta_{\pi_{\min}}$. Moreover it follows that $\Upsilon(T_1)=T_1(\pi_{\min})$, so that along our given subsequence we have $\limsup_{n\ra\infty}\lambda_{N}\leq \lambda_{\min}$ by Proposition \ref{prop:killing rate subsequential limits of stopped process}. Since our choice of convergent subsequence is arbitrary, it follows (using Lemma \ref{lem:liminf of lambda N}) that along the whole sequence we have
\begin{equation}\label{eq:convergence of varpiN}
\varpi^N\ra \pi_{\min}\quad\text{in}\quad \calP(\Rm_{>0})\quad\text{in probability and}\quad \lambda_N\ra \lambda_{\min}\quad \text{as}\quad N\ra\infty. 
\end{equation}

We write $\calM_{V}(\Rm_{>0})$ for the space of signed, finite Borel measures on $\Rm_{>0}$ equipped with the vague topology (induced by $C_c(\Rm_{>0})$ test functions). It follows from Proposition \ref{prop:relation between piN and varpiN}, \eqref{eq:convergence of varpiN}, and the fact that $G(C_c(\Rm_{>0}))\subseteq C_b(\Rm_{>0})$, that we have
\[
\xi^N=\lambda_N\varpi^NG\ra \lambda_{\min}\pi_{\min}G=\pi_{\min} \quad\text{in}\quad \calM_{V}(\Rm_{>0})\quad\text{as}\quad N\ra\infty. 
\]
Since there is no loss of mass - i.e. $\chi^N$ for $N\geq 12$ and $\pi_{\min}$ are all probability measures - it follows that 
\begin{equation}\label{eq:piN converges to pi}
\xi^N\ra \pi_{\min}\quad\text{in}\quad \calP(\Rm_{>0})\quad \text{as}\quad N\ra\infty.
\end{equation}
In particular, $\{\xi^N:N\geq 12\}$ is tight in $\calP(\Rm_{>0})$, hence $\{\chi^N:N\geq 12\}$ is tight in $\calP(\calP(\Rm_{>0}))$. 

We write $\chi$ for an arbitrary subsequetial limit in $\calP(\calP(\Rm_{>0}))$ of $\chi^N$. Then $\chi$ is $(\theta_t)_{t\geq 0}$-invariant by Proposition \ref{prop:subsequential limits are invariant measures}, and $\pi_{\min}$ is the mean measure of $\chi$ by \eqref{eq:piN converges to pi},
\[
\expE_{m\sim\chi}[m(\cdot)]=\pi_{\min}(\cdot).
\]
Since
\[
\expE_{m\sim\chi}\Big[\int_{\Rm_{>0}}e^{ux}m(dx)\Big]=\pi_{\min}(e^{ux})<\infty\quad\text{for all}\quad u<1,
\]
it follows that
\[
\int_{\Rm_{>0}}e^{ux}m(dx)<\infty\quad\text{for all $u<1$, for $\chi$-almost every $m\in \calP(\Rm_{>0})$}.
\]
Therefore $\chi$-almost every $m\in\calP(\Rm_{>0})$ is contained in the domain of attraction of $\pi_{\min}$ by Theorem \ref{theo:domain of attraction of minimal QSD}. Since $\chi$ is $(\theta_t)_{t\geq 0}$-invariant, it follows that $\chi=\delta_{\pi_{\min}}$.

This concludes the proof of Theorem \ref{theo:main theorem}.
\qed

\begin{appendix}
\section{Proof of Theorem \ref{theo:domain of attraction of minimal QSD}}\label{appendix:domain of attraction proof}

The implication that \eqref{eq:Martinez condition for domain of attraction} implies \eqref{eq:convergence to minimal QSD} is \cite[Theorem 1.3]{Martinez1998}. The equivalence of \eqref{eq:Martinez condition for domain of attraction} and \eqref{eq:condition for domain of attraction of minimal QSD} is immediate. To see that \eqref{eq:convergence to minimal QSD} implies \eqref{eq:limsup of survival prob bded} we calculate
\[
\frac{1}{n}\ln\Pm_{\mu}(\tau_{\partial}>n)=\frac{1}{n}\sum_{m<n}\ln\Pm_{\mu}(\tau_{\partial}>m+1\lvert \tau_{\partial}>m).
\]
Then since \eqref{eq:convergence to minimal QSD} implies that $\Pm_{\mu}(\tau_{\partial}>m+1\lvert \tau_{\partial}>m)\ra e^{-\lambda(\pi_{\min})}=e^{-\frac{1}{2}}$ as $m\ra\infty$, \eqref{eq:limsup of survival prob bded} follows.

Finally, we show that \eqref{eq:limsup of survival prob bded} implies \eqref{eq:condition for domain of attraction of minimal QSD}. We fix arbitrary $\epsilon\in (0,\frac{1}{4})$. It follows from \eqref{eq:limsup of survival prob bded} that
\[
\Pm_{\mu}(\tau_{\partial}>t)\leq e^{(-\frac{1}{2}+\epsilon)t}\quad \text{for all $t$ large enough.}
\]
Then defining $z:=\frac{1}{2}-2\epsilon$, it follows that $\expE_{\mu}[e^{z\tau_{\partial}}]<\infty$. It therefore follows from \cite[(1.4)]{Martinez1998} that
\[
\int_{\Rm_{>0}}e^{(1-2\sqrt{\epsilon})x}\mu(dx)=\int_{\Rm_{>0}}e^{-x(\sqrt{1-2z}-1)}\mu(dx)=\expE_{\mu}[e^{z\tau_{\partial}}]<\infty.
\]
Since $\epsilon\in (0,\frac{1}{4})$ is arbitrary, we obtain \eqref{eq:condition for domain of attraction of minimal QSD}.
\qed

\section{Proof of Theorem \ref{theo:ergodicity for fixed N} and Lemma \ref{lem:bound on number of jumps while mass all at the boundary}}\label{appendix:proof of N-particle ergodicity and jump bound}

We will begin by proving Theorem \ref{theo:ergodicity for fixed N} except for its final part, namely the assertions that $\lambda_N<\infty$ for $N\geq 12$ and that the convergence in \eqref{eq:almost sure convergence of lambda N fixed N} becomes $L^p$ convergence whenever the initial condition is deterministic and $1\leq p<\frac{1}{2}\lfloor \frac{N}{4}\rfloor$. We will then prove Lemma \ref{lem:bound on number of jumps while mass all at the boundary}. Finally we will conclude with the proof of the final part of Theorem \ref{theo:ergodicity for fixed N}.

\subsection{Proof of \eqref{eq:convergence in TV of Law fixed N} and \eqref{eq:almost sure convergence of lambda N fixed N}}

We fix $2\leq N<\infty$. For $10\leq H<\infty$ to be determined we consider the $N$-particle Fleming-Viot process $\vec{X}_t:=(X^1_t,\ldots,X^N_t)$ over the time interval $[0,H]$, suppressing the $N$ superscript as $N$ is fixed. We write ${\bf P}_t$ for the time $t$ transition kernel of the Fleming-Viot $N$-particle system, for $t\geq 0$. We claim that for some choice of $H<\infty$ there exists $\gamma>0$, $c_0\in (0,1)$ and $K<\infty$ such that
\begin{equation}\label{eq:Harris theorem condition}
{\bf P}_HV(\vec{X})\leq c_0V(\vec{X})+K\quad\text{for all}\quad \vec{X}\in \Rm_{>0}^N,
\end{equation}
whereby we define
\[
V(\vec{X})=V((X^1,\ldots,X^N)):=e^{\gamma \max(X^1,\ldots,X^N)}.
\]
\begin{proof}[Proof of \eqref{eq:Harris theorem condition}]
We denote the driving Brownian motions of $X^1_t,\ldots,X^N_t$ by $W^1_t,\ldots,W^N_t$ respectively. We define
\[
U^i_t:=X^i_t+t,\quad 1\leq i\leq N,\quad \bar U_t:=\max\{U^1_t,\ldots,U^N_t\}.
\]
We note that the path of $\bar U_t$ is continuous and follows the paths of $\{U^i_t\}$, at any time following the path of the maximal $U^i_t$. Moreover $U^i_t$ satisfies $dU^i_t=dW^i_t$ between the jump times of $X^i_t$. 

We fix $y>x:=\bar U_0$ for the time being and take $z\in (x,y)$ to be determined. For some $M<\infty$ such that 
\[
x<z<y-M(N+5)
\] 
we assume for contradiction that:
\begin{enumerate}
\item\label{enum:assum for contradiction bounded oscillation}
the oscillations of the driving Brownian motions are bounded by $M$, $\sup_{1\leq i\leq N}[\sup_{t\leq H}W^i_t-\inf_{t\leq H}W^i_t]\leq M$;
\item\label{enum:assum for contradiction at least z}
no particle with $U^i_t\geq z$ for some $0\leq t\leq H$ then dies in the time interval $[t,H]$.
\end{enumerate}

We take $i_0\in \{1,\ldots,N\}$ such that $U^{i_0}_H=\bar U_H$. Then for some $t_0\leq H$, particle $X^{i_0}$ must have died at time $t_0$ and jumped onto a particle $X^{i_1}$ with $U^{i_1}_{t_0}\geq y-M$, since otherwise $U^{i_0}_{t}$ would have travelled from strictly below $y-M$ at some time $t'<H$ say (note that $U^i_0\leq \bar U_0= x<z<y-M$), to at least $y$ at time $H$ without $X^{i_0}$ being killed, which would then imply $\lvert W^{i_0}_H-W^{i_0}_{t'}\rvert >M$. The time $t_0$ must have been the first time that $U^{i_0}_t\geq z$, by \ref{enum:assum for contradiction at least z}. 

We then proceed inductively, tracing $U^{i_n}_t$ backwards from time $t_{n-1}$ until the time $t_n$ at which $X^{i_n}$ jumped onto some particle $X^{i_{n+1}}$ with $U^{i_{n+1}}_{t_n}\geq y-(n+1)M$. This must have been the first time that $U^{i_n}_t\geq z$, again by \ref{enum:assum for contradiction at least z}, as long as $n\leq N+3$.

We terminate our induction procedure when either we reach time $0$ or $n=N+2$. We set the $n$ at which we terminate to be $n_{\max}$. We see that the path we obtain cannot go below $y-(n_{\max}+2)M>\bar U_0$, so we must have terminated when $n_{\max}=N+2$. On the other hand, since for each $i_n$ the time $t_n$ was the first time that $U^{i_n}\geq z$, it follows that there can be no repeated indices, so we cannot have $n_{\max}=N+2$. This is a contradiction.

It follows that if we have \ref{enum:assum for contradiction bounded oscillation} and \ref{enum:assum for contradiction at least z}, then $\bar U_H\leq y$. Moreover, we observe that if $z>M+H$ then \ref{enum:assum for contradiction bounded oscillation} implies \ref{enum:assum for contradiction at least z}. Therefore if we have $z>0$ and $M<\infty$ such that
\begin{equation}\label{eq:sufficient condition for M}
\bar U_0\vee (M+H) < z<y-M(N+5)<y,
\end{equation}
then \ref{enum:assum for contradiction bounded oscillation} implies $\bar U_H<y$, so that $\bar U_H\geq y$ implies \ref{enum:assum for contradiction bounded oscillation}$^c$ (the complement of the event defined by \ref{enum:assum for contradiction bounded oscillation}). 

We fix the initial condition $\vec{X}_0\in \Rm_{>0}^N$; in the following none of the constants depend upon our choice of $\vec{X}_0\in \Rm_{>0}^N$. We can bound the probability of \ref{enum:assum for contradiction bounded oscillation}$^c$ explicitly by the reflection principle,
\[
\Pm(\ref{enum:assum for contradiction bounded oscillation}^c)\leq 4N\Phi(\frac{M}{2\sqrt{H}}),
\]
whereby $\Phi(a):=\Pm(N(0,1)\geq a)$. We see that if $y>x\vee H$, then
\begin{equation}\label{eq:M(y) formula}
M=M(x,y):=\frac{y-(x\vee H)}{5N}\leq \Big(\frac{y-H}{N+6}\Big)\wedge \Big(\frac{y-x}{N+5}\Big)
\end{equation}
suffices to provide for \ref{eq:sufficient condition for M}. We note that $M(x,y)=M(x\vee H,y)$. Therefore we have
\begin{equation}
\Pm(\bar U_H\geq y)\leq 4N\Phi\Big(\frac{M(\bar U_0\vee H,y) }{2\sqrt{H}}\Big)\quad \text{for all}\quad y\geq \bar U_0\vee H.
\end{equation}

We now define $\bar X_t:=\max(X^1_t,\ldots,X^N_t)$, and note that $\bar X_0=\bar U_0$, $\bar X_H=\bar U_H-H$. It follows that
\begin{equation}
\Pm(\bar X_H\geq y-H)\leq 4N\Phi\Big(\frac{M(\bar X_0\vee H,y) }{2\sqrt{H}}\Big)\quad \text{for all}\quad y\geq \bar X_0\vee H.
\end{equation}
We take $\gamma=\frac{2\ln 2}{H}>0$ (so that $e^{-\frac{\gamma H}{2}}=\frac{1}{2}$) and $5\leq C<\infty$ to be determined. We calculate 
\begin{equation}\label{eq:first bound for e to gamma H}
\begin{split}
\expE[e^{\gamma \bar X_H}]\leq \expE[e^{\gamma \bar X_H}\Ind(\bar X_H\leq \bar X_0-\frac{H}{2})]\\+\expE[e^{\gamma \bar X_H}\Ind(\bar X_0-\frac{H}{2}<\bar X_H\leq \bar X_0+(C+5)H)]+\expE[e^{\gamma \bar X_H}\Ind(\bar X_H>\bar X_0+(C+5)H)]\\
\leq \frac{1}{2}e^{\gamma (\bar X_0\vee H)}+4^{C+5}e^{\gamma(\bar X_0\vee H)}\Pm(\bar X_H\geq \bar X_0-\frac{H}{2}) +\expE[e^{\gamma \bar X_H}\Ind(\bar X_H>\bar X_0+(C+5)H)].
\end{split}
\end{equation}

We assume for the time being that $\bar X_0\geq H$ and seek to bound the right hand side of \eqref{eq:first bound for e to gamma H}. In the following, $B<\infty$ and $b>0$ are constants which are uniform in $H$, and which may increase (respectively decrease) from line to line. Using the bound $\Phi(a)\leq \frac{1}{a}e^{\frac{-a^2}{2}}$ for $a\geq 1$, we calculate that
\[
\Pm(\bar X_H\geq \bar X_0-\frac{H}{2})\leq B\Phi\Big(\frac{M(\bar X_0,\bar X_0+\frac{H}{2})}{2\sqrt{H}}\Big)\leq \frac{B\sqrt{H}}{M(\bar X_0,\bar X_0+\frac{H}{2})}e^{-b\frac{[M(\bar X_0,\bar X_0+\frac{H}{2})]^2}{ H}}\\
\leq \frac{B}{\sqrt{H}}e^{-b H }.
\]
We see that for $r>0$ sufficiently small, defining $C=C(H):=rH$ gives that
\[
4^{C+5}\Pm(\bar X_H\geq \bar X_0-\frac{H}{2})<\frac{1}{10}
\]
for all $H$ large enough. We fix such an $r$, thereby defining $C=C(H)$.

We calculate that
\[
\begin{split}
\expE[e^{\gamma \bar X_H}\Ind(\bar X_H>\bar X_0+(C+5)H)]\leq \sum_{n\geq 0}e^{\gamma [\bar X_0+(C+5)H+n+1]}\Pm(\bar X_H\geq \bar X_0+(C+5)H+n)\\
\leq \int_{\bar X_0+(C+2)H}^{\infty}e^{10\gamma H}e^{\gamma y}\Pm(\bar X_H\geq y-H)dy\leq B\int_{\bar X_0+(C+2)H}^{\infty}e^{\gamma y}\Phi\Big(\frac{M(\bar X_0,y)}{2\sqrt{H}}\Big)dy\\
\leq B\sqrt{H}\int_{\bar X_0+(C+2)H}^{\infty}\frac{e^{\gamma y}e^{-\frac{b(M(\bar X_0,y))^2}{H}}}{M(\bar X_0,y)} dy\leq \frac{B}{H^{\frac{3}{2}}}\int_{\bar X_0+(C+2)H}^{\infty} e^{\gamma y}e^{-\frac{b(y-\bar X_0)^2}{H}}  dy\\
\leq  \frac{B}{ H^{\frac{3}{2}}}e^{\gamma \bar X_0}\int_{rH^2}^{\infty}e^{\gamma z}e^{-\frac{bz^2}{ H }}dz=\frac{B}{ \sqrt{H}}e^{\gamma \bar X_0}\int_{rH }^{\infty}e^{\gamma H v}e^{-bHv^2 }dv\leq \frac{B}{ \sqrt{H}}e^{\gamma \bar X_0}\int_{rH }^{\infty}e^{(B-bHv)v  }dv.
\end{split}
\]

We conclude from \eqref{eq:first bound for e to gamma H} that for all $H$ sufficiently large we have $\expE[e^{\gamma \bar X_H}]\leq \frac{3}{4}e^{\gamma \bar X_0}$ if $\bar X_0\geq H$. If $\bar X_0\leq H$, we similarly conclude that $\expE[e^{\gamma \bar X_H}]\leq \frac{3}{4}e^{\gamma H}$. We therefore obtain \eqref{eq:Harris theorem condition}.
\end{proof}

We take $c_0$ and $K$ as given by \eqref{eq:Harris theorem condition}. We take $R>\frac{2K}{1-c_0}$ and $\mathscr{C}:=\{\vec{X}\in \Rm_{>0}^N:V(\vec{X})\leq R\}$. We claim that there exists $\alpha>0$ and $\nu\in \calP(\Rm_{>0})$ such that
\begin{equation}\label{eq:Doeblin condition for Harris theorem}
{\bf P}_H(\vec{X},\cdot)\geq \alpha \nu(\cdot)\quad\text{for all}\quad \vec{X}\in \mathscr{C}.
\end{equation}

We assume that $\vec{X}_0\in \mathscr{C}$. We recall that $\bar X_t:=\max(X^1_t,\ldots,X^N_t)$ and define $\hat{X}_t:= \sum_{i=1}^NX^i_t$. We further define $B:=\frac{\ln R}{\gamma}$. Then $\hat X_0\leq NB$ whenever $\vec{X}_0\in \mathscr{C}$. We take $0<\epsilon<B$ and define $\tau_{\epsilon}:=\inf\{t>0:\bar X_t\geq \epsilon\}$. We see that $\hat X_t$ satisfies $d\hat X_t=-Ndt+\sqrt{N}d\hat{W}_t+dL_t$ for some Brownian motion $\tilde{W}_t$ and non-decreasing process $L_t$. We see that if $\hat{W}_{\frac{H}{3}}-\hat{W}_0\geq \frac{NH}{3}+N\epsilon$, then $\hat{X}_{\frac{H}{3}}\geq N\epsilon$, so that $\bar X_{\frac{H}{3}}\geq \epsilon$, hence $\tau_{\epsilon}\leq \frac{H}{3}$. Therefore there is a probability bounded away from $0$ that $\tau_{\epsilon}\leq \frac{H}{3}$.

In this case, we must have $\bar{X}_{\tau_{\epsilon}}\in [\epsilon,B]$ and $\vec{X}_{\tau_{\epsilon}}\in [0,B]^N$. There is then a probability uniformly bounded away from $0$ that from time $\tau_{\epsilon}$ to time $\frac{2H}{3}$ the maximal particle ($X^{1}$ say) stays in $[\frac{\epsilon}{2},2B]$, whilst all other particles die exactly once, jump onto $X^1$, and stay within $[\frac{\epsilon}{3},3B]$. In this case all particles are contained in $[\frac{\epsilon}{3},3B]$ at time $\frac{2H}{3}$. We then obtain \eqref{eq:Doeblin condition for Harris theorem} from the parabolic Harnack inequality.

It therefore follows from Harris' ergodic theorem \cite[Theorem 1.2]{Hairer2011} that ${\bf P}_H$ has a unique stationary distribution, $\psi\in \calP(\Rm_{>0}^N)$, such that ${\bf P}_{nH}(\upsilon,\cdot)\ra \psi(\cdot)$ in total variation as $n\ra\infty$, for all $\upsilon\in \calP(\Rm_{>0}^N)$. We fix arbitrary $t\geq 0$. Since $\psi$ is the unique stationary distribution for ${\bf P}_H$, which commutes with $ {\bf P}_t$, it follows from $(\psi {\bf P}_t)=\psi{\bf P}_H{\bf P}_t=(\psi{\bf P}_t){\bf P}_H$ that $\psi {\bf P}_t=\psi$. Therefore $\psi$ is stationary for $( {\bf P}_t)_{t\geq 0}$, and we have that
\[
\upsilon{\bf P}_t=(\upsilon {\bf P}_{\lfloor \frac{t}{H}\rfloor H}-\psi){\bf P}_{t-\lfloor \frac{t}{H}\rfloor H}+\psi {\bf P}_{t-\lfloor \frac{t}{H}\rfloor H}\overset{TV}{\ra} \psi\quad\text{as}\quad t\ra\infty,
\]
for all $\upsilon \in \calP(\Rm_{>0}^N)$. We have therefore established \eqref{eq:convergence in TV of Law fixed N}.

We obtain \eqref{eq:almost sure convergence of lambda N fixed N} for some $\lambda_N\in [0,\infty]$ from \eqref{eq:convergence in TV of Law fixed N} and Birkhoff's ergodic theorem. We obtain that $\lambda_N>0$ simply by observing that, started from the stationary distribution, the expected number of jumps in time $1$ must be strictly positive.
\qed

\subsection{Proof of Lemma \ref{lem:bound on number of jumps while mass all at the boundary}}

We begin by identifying
\begin{equation}\label{eq:shift time horizon by 1 stationary}
\expE_{\vec{X}^N_0\sim \psi^N}\Big[\frac{1}{N}\sum_{\tau_n\leq 1}\Ind(m^N_{\tau_n}(B(0,\delta))>\epsilon)\Big]=\expE_{\vec{X}^N_0\sim \psi^N}\Big[\frac{1}{N}\sum_{1\leq \tau_n\leq 2}\Ind(m^N_{\tau_n}(B(0,\delta))>\epsilon)\Big].
\end{equation}
Our goal will be to control the latter. We fix arbitrary $\epsilon>0$ and take $\delta>0$ to be determined.

We firstly define $\bar X^N_t:=\max_{1\leq i\leq N}X^{N,i}_t$. We then define
\[
\begin{split}
\hat{J}^{N,k}_t:=\sum_{\tau_n\leq t}\Ind(2^{-k}\leq \bar X^N_{\tau_n}),\quad k\geq 0.
\end{split}
\]
We observe that we are not renormalising the number of jumps by $N$ in the above quantity, and that the maximum $\bar X^N_t$ cannot change at jump times. We further define the stopping times
\[
\tau^N_{\delta}:=\inf\{t\geq 1:m^N_t(B(0,\delta))>\epsilon\}, \quad \delta>0,\quad \tau^N_k:=\inf\{t\geq 1:\bar X^N_{t}< 2^{-k+1}\},\quad k\geq 1.
\]
We now take the bound
\[
\sum_{1\leq \tau_n^N\leq 2}\Ind[m^N_{\tau_n}(B(0,\delta))>\epsilon]\leq\sum_{k\geq 1} (\hat{J}^{N,k}_{2}-\hat{J}^{N,k}_{\tau^N_k})\Ind(\tau^N_k\leq 2)+(\hat{J}^{N,0}_{2}-\hat{J}^{N,0}_{\tau^N_{\delta}})\Ind(\tau^N_{\delta}\leq 2).
\]
If a particle dies while $\bar X^N_t\geq 2^{-k}$, there is a $\frac{1}{N-1}$ probability that it jumps onto the maximal particle. If that happens, in order for it to then die again within time $2^{-2k-1}$, its driving Brownian motion must travel a distance at least $2^{-k-1}$. Thus there is a probability $p>0$, uniform in $k$, that this does not happen. Therefore the number of deaths of a given particle in time $2^{-2k-1}$ while $\bar X^N_t\geq 2^{-k}$ is stochastically dominated by $1+\text{Geom}(\frac{p}{N-1})$. We decompose the interval $[1,2]$ into $2^{2k+1}$ subintervals, and take this bound on each subinterval and all $N$ particles. We conclude that there exists $C<\infty$ such that 
\begin{equation}\label{eq:number of jumps conditional on stopping time controls}
\expE[\hat{J}^{N,k}_{2}-\hat{J}^{N,k}_{\tau^N_k}\lvert \tau^N_k\leq 2]\leq CN^22^{2k}\quad\text{for all}\quad  k\geq 1,\quad \expE[\hat{J}^{N,0}_{2}-\hat{J}^{N,0}_{\tau^N_{\delta}}\lvert \tau^N_{\delta}\leq 2]\leq CN^2.
\end{equation}
Therefore we conclude that 
\begin{equation}\label{eq:bound on expectation jumps in ball estimate}
\expE_{\vec{X}^N_0\sim \psi^N}\Big[\frac{1}{N}\sum_{1\leq \tau_n\leq 2}\Ind[m^N_{\tau_n}(B(0,\delta))>\epsilon]\Big]\leq CN^2\sum_{k\geq 0} 2^{2k}\Pm(\tau^N_k\leq 2)+CN^2\Pm(\tau^N_{\delta}\leq 2).
\end{equation}

Our goal is to establish controls on the probabilities in \eqref{eq:bound on expectation jumps in ball estimate}. To this end we construct the following coupling. We denote as $W^{N,1}_t,\ldots,W^{N,N}_t$ the driving Brownian motions of $X^{N,1}_t,\ldots,X^{N,N}_t$ respectively. Using these Brownian motions, we then take strong solutions (which exist by \cite[Theorem 2.1]{Lions1984}) of
\begin{equation}\label{eq:SDE for process Y}
dY^{N,i}_t=-dt+dW^{N,i}_t+dL^{N,i}_t,\quad 0\leq t<\infty,\quad Y^{N,i}_0=0,\quad 1\leq i\leq N,
\end{equation}
where $L^{N,i}_t$ is the local time of $Y^{N,i}_t$ at $0$. Thus $Y^{N,i}_t$ is a Brownian motion with drift $-1$, reflected at $0$. Therefore $\vec{Y}^{N}=(Y^{N,1},\ldots,Y^{N,N})$ and $\vec{X}^{N}$ are coupled so that
\[
X^{N,i}_t\geq Y^{N,i}_t\quad\text{for all}\quad 0\leq t<\infty,\quad 1\leq i\leq N.
\]
We note in particular that $Y^{N,1}_t,\ldots,Y^{N,N}_t$ are jointly independent. 

We now establish that there exists $C<\infty$ and $q<1$ such that we have
\begin{equation}\label{eq:bound on tkN prob}
\Pm(\tau_k^N\leq 2)\leq [C2^{-k}\wedge q]^{\lfloor\frac{N}{4}\rfloor},\quad \text{for all}\quad k\geq 1,\quad N\geq 4.
\end{equation}

We define for $N\geq 4$,
\[
D^{N,\ell}_t:=\max(Y^{N,4\ell-3}_t,Y^{N,4\ell-2}_t,Y^{N,4\ell-1}_t,Y^{N,4\ell}_t),\quad \bar D^{N,\ell}:=\inf_{1\leq t\leq 2}D^{N,\ell}_t,\quad 1\leq \ell\leq \lfloor \frac{N}{4}\rfloor.
\]
We observe that
\begin{equation}\label{eq:bound on tkN and tepsilonN probabilities}
\begin{split}
\{\tau^N_{k }\leq 2\}\subseteq \{\bar D^N_{\ell}\leq 2^{-k+1}\quad\text{for all}\quad 1\leq \ell\leq \lfloor\frac{N}{4}\rfloor\},\quad \text{for all}\quad k\geq 1.
\end{split}
\end{equation}

It follows from the joint independence of $\{\bar D^{N,\ell}:1\leq \ell\leq \lfloor \frac{N}{4}\rfloor\}$ that establishing the existence of $C<\infty$ and $q<1$ such that for all $N\geq 4$ and $k\geq 0$ we have
\begin{equation}\label{eq:bound on prob of max D small event}
\Pm(\bar D^{N,\ell}\leq 2^{-k})\leq C2^{-k}\wedge q,\quad 1\leq \ell \leq \lfloor\frac{N}{4}\rfloor,
\end{equation}
suffices to give \eqref{eq:bound on tkN prob}. We will now establish \eqref{eq:bound on prob of max D small event}.

\begin{proof}[Proof of \eqref{eq:bound on prob of max D small event}]
We assume without loss of generality that $\ell=1$. We define 
\[
f(\vec{y})=f(y_1,y_2,y_3,y_4):=\frac{1}{\sqrt{y_1^2+y_2^2+y_3^2+y_4^2}},\quad \vec{y}\in \Rm^4\setminus \{0\}.
\]
We consider $(\vec{Y}_t)_{0\leq t\leq 2}:=(Y^1_t,Y^2_t,Y^3_t,Y^4_t)_{0\leq t\leq 2}$ where $Y^1_t,Y^2_t,Y^3_t,Y^4_t$ are independent copies of the SDE \eqref{eq:SDE for process Y}, with independent driving Brownian motions $W^1_t,W^2_t,W^3_t,W^4_t$ respectively. Since $(\vec{Y}_t)_{0\leq t\leq 2}$ has a bounded $4$-dimensional Lebesgue density at time $1$, we have that $\expE[f(\vec{Y}_0)]<\infty$. We now observe by Ito's formula that
\[
df(\vec{Y}_t)=f^3(\vec{Y}_t)\Big[ \sum_{i=1}^4Y^i_t-\frac{1}{2} \Big]dt- f^3(\vec{Y}_t)\Big(\sum_{i=1}^4Y^i_tdW^i_t\Big)-dL_t,
\]
whereby $L_t$ is some non-decreasing process. It follows that there exists $C'<\infty$ such that $f(\vec{Y}_t)-C't$ is a supermartingale. Moreover we observe that if $\max(y_1,\ldots,y_4)\leq 2^{-k}$ then $f(\vec{y})\geq 2^{k-1}$. Therefore
\begin{equation}
\Pm(\bar D^{N,1}\leq 2^{-k})\leq \Pm(\sup_{1\leq t\leq 2}f(\vec{Y}_t)\geq 2^{k-1})\leq \Pm(\sup_{1\leq t\leq 2}[f(\vec{Y}_t)-C'(t-2)]\geq 2^{k-1})\\
\leq \frac{\expE[f(\vec{Y}_1)]+C'}{2^{k-1}}.
\end{equation}
We therefore have \eqref{eq:bound on prob of max D small event} for all $k$ such that $2^{k-1}>2(C'+\expE[f(\vec{Y}_1)])$. On the other hand, $\Pm(\bar D^{N,1}\leq 1)<1$, so we have \eqref{eq:bound on prob of max D small event} for all $k$.
\end{proof}
Having established \eqref{eq:bound on prob of max D small event}, we have now established \eqref{eq:bound on tkN prob}.

We now seek to show that for some $\delta>0$ there exists $N_0<\infty$ and $\gamma>0$ such that
\begin{equation}\label{eq:bound on tepsilonN prob}
\Pm(\tau^N_{\delta}\leq 2)\leq e^{-\gamma N},\quad \text{for all}\quad N_0\leq N<\infty.
\end{equation}
We take $n_0\in \Nm$ such that $\frac{1}{n_0}<\frac{\epsilon}{10}$, and henceforth assume that $N\geq n_0$ ($N_0$ has not yet been determined). We define $U^{\ell}_t$ to be the value of the second smallest of $Y^{n_0(\ell-1)+1}_t,\ldots,Y^{n_0\ell}_t$ for $1\leq \ell\leq \lfloor\frac{N}{n_0}\rfloor$. Then we have that
\[
\sup_{1\leq t\leq 2}m^N_t(B(0,\delta))\leq \frac{1}{n_0}+\frac{n_0}{N}+ \frac{n_0}{N}\sum_{\ell=1}^{\lfloor\frac{N}{n_0}\rfloor}\Ind(\inf_{1\leq t\leq 2}U_t^{\ell}\leq \delta).
\]
Since no two particles can hit the boundary at the same time, we can choose $\delta>0$ such that $\Pm(\inf_{1\leq t\leq 2}U_t^{\ell}\leq \delta)\leq \frac{\epsilon}{10}$. We therefore obtain \eqref{eq:bound on tepsilonN prob} by Cramer's theorem.

Therefore, inputting \eqref{eq:bound on tkN prob} and \eqref{eq:bound on tepsilonN prob} into \eqref{eq:bound on expectation jumps in ball estimate} we conclude that for some $C,N_0<\infty$ we have for all $N\geq N_0\wedge 4$ that
\[
\expE_{\vec{X}^N_0\sim \psi^N}\Big[\frac{1}{N}\sum_{1\leq \tau_n\leq 2}\Ind[m^N_{\tau_n}(B(0,\delta))>1-\epsilon]\Big]\leq CN^2\sum_{k\geq 1} 2^{2k}[C2^{-k}\wedge q]^{\lfloor\frac{N}{4}\rfloor}+2CN^22e^{-\gamma N}.
\]
This concludes the proof of Lemma \ref{lem:bound on number of jumps while mass all at the boundary}.
\qed

\subsection{Proof of the final part of Theorem \ref{theo:ergodicity for fixed N}}

As in \eqref{eq:shift time horizon by 1 stationary} we have
\[
\lambda_N=\expE_{\vec{X}_0\sim \psi^N}[J^N_1-J^N_0]=\expE_{\vec{X}_0\sim \psi^N}[J^N_2-J^N_1].
\]

We have from \eqref{eq:number of jumps conditional on stopping time controls} and \eqref{eq:bound on tkN prob} that
\[
\expE_{\vec{X}_0\sim \psi^N}[\sum_{1\leq \tau_n\leq 2}\Ind(\max_{1\leq i\leq N}X^{N,i}_{\tau_n}\leq 1)]\leq \sum_{k\geq 0}CN^22^{2k}[C2^{-k}\wedge q]^{\lfloor\frac{N}{4}\rfloor},
\]
for some uniform constant $C<\infty$. It follows that 
\[
\expE_{\vec{X}_0\sim \psi^N}[\sum_{1\leq \tau_n\leq 2}\Ind(\max_{1\leq i\leq N}X^{N,i}_{\tau_n}\leq 1)]<\infty\quad\text{for}\quad N\geq 12.
\]
On the other hand, if a particle dies when the maximum particle is at least $1$, then there is a probability bounded away from $0$ that it does not die again in time $1$. Therefore the number of times a particle is killed over the time horizon $[1,2]$ while the maximum is at least $1$ is stochastically dominated by a $\text{Geom}(p)$ random variable, for some $p\in (0,1)$. It follows that $\lambda_N<\infty$ for $N\geq 12$.

We finally turn to the proof that the convergence in \eqref{eq:almost sure convergence of lambda N fixed N} becomes $L^p$ convergence whenever the initial condition is deterministic and $1\leq p<\frac{1}{2}\lfloor \frac{N}{4}\rfloor$. It suffices to establish that $\sup_{t\geq 1}\lvert \lvert \frac{J^N_t}{t}\rvert\rvert_{L^p}$ is finite for any such $p$, since $L^{p'}$-boundedness for some $p'\in (p,\frac{1}{2}\lfloor \frac{N}{4}\rfloor)$ and almost sure convergence implies $L^p$ convergence.

We note that the controls on the number of jumps between times $1$ and $2$ obtained in the proof of Lemma \ref{lem:bound on number of jumps while mass all at the boundary} don't depend upon the choice of initial condition, so apply for any initial condition and any time interval $[n,n+1]$, for $n\geq 1$. Moreover, since the initial condition is deterministic, we can take the initial condition for the  processes defined in \eqref{eq:SDE for process Y} to instead be $Y^{N,i}_0=\epsilon$ for some $\epsilon>0$ sufficiently small, allowing us to obtain the same controls over the time horizon $[0,1]$. 

Replacing the $L^1$ controls with $L^p$ controls obtained in exactly the same manner, we see that there exists $C<\infty$ and $q<1$ such that
\[
\Big\lvert\Big\lvert \frac{J^N_t}{t}\Big\rvert\Big\rvert_{L^p}\leq C\Big[1+\sum_{k\geq 0}2^{2k}\Big([C2^{-k}\wedge q]^{\lfloor \frac{N}{4}\rfloor}\Big)^{\frac{1}{p}}\Big]\quad\text{for all}\quad 1\leq t<\infty.
\]
This is finite when $\frac{1}{p}\lfloor \frac{N}{4}\rfloor>2$.
\qed

\section{Statement of Theorem \ref{theo:tight ic hydrodynamic limit theorem}}\label{appendix:tight ic hydrodynamic limit theorem}

We equip $\mathcal{P}(\Rm_{>0})$ with the Wasserstein-$1$ metric on $\mathcal{P}$ using the bounded metric $d^1(x,y):= 1 \wedge \lvert x-y\rvert$ on the underlying space $\Rm_{>0}$, which metrises the topology of weak convergence of measures \cite{Gibbs2002}. We write $\mathcal{P}_{\Wah}(\Rm_{>0})$ for this metric space. We then define on $\mathcal{D}([0,\infty);\mathcal{P}_{\Wah}(\Rm_{>0})\times \Rm_{\geq 0})$ the metric
\begin{equation}
\begin{split}
d^{\mathcal{D}}(f,g):=\sum_{T=1}^{\infty}2^{-T}(d_{\mathcal{D}([0,T];\mathcal{P}_{\Wah}(\Rm_{> 0})\times \Rm_{\geq 0})}((f_t)_{0\leq t\leq T},(g_t)_{0\leq t\leq T})\wedge 1).
\end{split}
\label{eq:Skorokhod metric on integer times}
\end{equation}
We can now state the following hydrodynamic limit theorem, \cite[Theorem 2.10]{Tough2022}, which we employ in the proof of Theorem \ref{theo:main theorem}.

\begin{theo}[Theorem 2.10, \cite{Tough2022}]\label{theo:tight ic hydrodynamic limit theorem}
We consider a sequence of Fleming-Viot processes $\{(\vec{X}^N_t)_{t\geq 0}:N\geq 12\}$. We define $(m^N_t)_{t\geq 0}$ and $(J^N_t)_{t\geq 0}$ as in \eqref{eq:jump emp meas processes}, and assume
that $\{\Law(m^N_0):N\geq 12\}$ is a tight family of measures in $\mathcal{P}(\mathcal{P}_{\Wah}(\Rm_{>0}))$. Then $\{\Law((m^N_t,J^N_t)_{0\leq t\leq T})\}$ is a tight family of measures in $\calP((\mathcal{D}([0,\infty);\mathcal{P}_{\Wah}(\Rm_{> 0})\times \Rm_{\geq 0}),d^{\mathcal{D}}))$. Moreover every subsequential limit is supported on
\[
\{(\mathcal{L}_{\mu}(X_t\lvert \tau_{\partial}>t), -\ln\Pm_{\mu}(\tau_{\partial}>t))_{0\leq t<\infty}\in \mathcal{C}([0,\infty);\mathcal{P}_{\Wah}(\Rm_{>0})\times \Rm_{\geq 0}):\mu\in \calP(\Rm_{>0})\}.
\]
\end{theo}

\end{appendix}

{\textbf{Acknowledgement:}} 
The author is grateful to James Nolen and Denis Villemonais for useful discussions on this problem. This work was partially supported by the EPSRC MathRad programme grant EP/W026899/.
\bibliography{SelectionPrincipleFleming-ViotProcess}
\bibliographystyle{plain}
\end{document}